\documentclass[a4paper,12p]{article}
\usepackage[english]{babel}
\usepackage[utf8]{inputenc}
\usepackage[T1]{fontenc}
\usepackage{amsmath}
\usepackage{amsthm}
\usepackage{amssymb}
\usepackage{amsfonts} 
\usepackage{fullpage}
\usepackage{cleveref}
\usepackage{url}
\usepackage{multirow}
\usepackage{todonotes}
\usepackage{subcaption}
\usepackage{verbatim}
\usepackage{wasysym}

\usepackage{tikz}
\usetikzlibrary{calc}
\tikzstyle{noeud}=[circle,inner sep=2, minimum size =3 pt, line width = 1pt, draw=black, fill=white]
\definecolor{bleu}{rgb}{0.15, 0.31, 0.86}
\definecolor{rouge}{RGB}{182, 16, 0}

\DeclareMathOperator{\ex}{ex} 
\DeclareMathOperator{\bal}{bal} 
\DeclareMathOperator{\lbal}{lbal} 

\newtheorem{theorem}{Theorem}
\newtheorem{lemma}[theorem]{Lemma}
\newtheorem{definition}[theorem]{Definition}
\newtheorem{prop}[theorem]{Proposition}
\newtheorem{corollary}[theorem]{Corollary}

\crefname{proposition}{proposition}{propositions}

\newtheorem{claim}{Claim}[theorem]

\crefname{claim}{claim}{claims}

\newcommand{\claimproof}[1]{\noindent\emph{Proof of Claim~\ref{#1}.} }
\newcommand{\claimqed}{\hfill{$\rhd$}}
\newcommand{\half}{\mathcal{H}}

\title{The balancing number and list balancing number of some graph classes\thanks{This work has been supported by PAPIIT IN111819, PAPIIT TA100820 and CONACyT project 282280.}}

\author{
	Antoine Dailly $^\ddagger$ 
	\and Laura Eslava $^\S$ 
	\and Adriana Hansberg $^\ddagger$ 
	\and Denae Ventura \thanks{Corresponding author} $^{,\ddagger}$ 
	\\ \\ \\
	$^\ddagger$ Instituto de Matem\'aticas, UNAM Juriquilla, 76230 Quer\'etaro, Mexico.
	\\
	$^\S$ IIMAS, UNAM Ciudad Universitaria, 04510 Mexico City, Mexico}
\date{}

\begin{document}
	
	\maketitle
	
	\begin{abstract}

		Given a graph $G$, a 2-coloring of the edges of $K_n$ is said to contain a \emph{balanced copy} of $G$ if we can find a copy of $G$ such that half of its edges is in each color class. If there exists an integer $k$ such that, for $n$ sufficiently large, every 2-coloring of $K_n$ with more than $k$ edges in each color contains a balanced copy of $G$, then we say that $G$ is \emph{balanceable}. The smallest integer $k$ such that this holds is called the \emph{balancing number of $G$}.
		%
		
		In this paper, we define a more general variant of the balancing number, the \emph{list balancing number}, by considering 2-list edge colorings of $K_n$, where every edge $e$ has an associated list $L(e)$ which is a nonempty subset of the color set $\{r,b\}$. In this case, edges $e$ with $L(e) = \{r,b\}$ act as jokers in the sense that their color can be chosen $r$ or $b$ as needed. In contrast to the balancing number, every graph has a list balancing number. Moreover, if the balancing number exists, then it coincides with the list balancing number.
		
		We give the exact value of the list balancing number for all cycles except for $4k$-cycles for which we give tight bounds. 
		In addition, we give general bounds for the list balancing number of non-balanceable graphs based on the extremal number of its subgraphs, and study the list balancing number of $K_5$, which turns out to be surprisingly large.
	\end{abstract}
	
	\section{Introduction}
	
Ramsey Theory studies the presence of ordered substructures in large, arbitrarily ordered structures. For instance, the seminal Ramsey Theorem~\cite{R30} states that, for every integer $r$, every 2-coloring of the edges of $K_n$ contains a monochromatic $K_r$ whenever $n$ is sufficiently large. However, it is also possible to look for other kinds of ordered substructures. In this line, Bollob\'as (see \cite{CuMo08}) conjectured in 2008 that, for any real $0 < \varepsilon \le \frac{1}{2}$, and for large enough $n$, any 2-coloring of $K_n$ with at least $\varepsilon \binom{n}{2}$ edges in each color contains a $K_{2t}$ where one of the colors induces a clique of size $t$, or two disjoint cliques of size $t$. This conjecture was confirmed by Cutler and Mont\'agh in \cite{CuMo08} and an asymptotically tight bound was obtained by Fox and Sudakov in \cite{FoSu08}. Making the connection to extremal problems in graphs, Caro, Hansberg and Montejano~\cite{CHM19} considered $\varepsilon$ as a function of $n$ tending to $0$, and reconfirmed Bollob\'as conjecture obtaining also a subquadratic bound on the number of edges that are required in each color. Recently, Gir\~{a}o and Narayanan~\cite{GiNa19} showed that $\Omega(n^{2-\frac{1}{t}})$ edges from each color are sufficient and, conditional on the K\H{o}vari-S\'os-Tur\'an conjecture for the Tur\'an number for complete bipartite graphs \cite{KST54}, they showed that this bound is sharp up to the involved constants. The notion of \emph{balanceability}, which was introduced in \cite{CHM19}, is enclosed within this setting, as it is concerned with finding copies of a given graph $G$ in 2-colorings of the edges of $K_n$, the only condition being that the copy contains half of its edges in each of the color classes $R$ and $B$.
	
	More formally, a 2-coloring of the edges of $K_n$ is a function $f:E(K_n) \rightarrow \{r,b\}$. We can see a 2-coloring of the edges of $K_n$ as a partition $E(K_n)=R \sqcup B$ where we define the color class $R$ (resp. $B$) as the set of edges $e$ such that $f(e)=r$ (resp. such that $f(e)=b$). Edges in $R$ are called \emph{red}, edges in $B$ are called \emph{blue}. Let $G(V,E)$ be a simple, finite graph; a 2-coloring $R \sqcup B$ of the edges of $K_n$ is said to contain a \emph{balanced copy} of $G$ if we can find a copy of $G$ such that its edge-set $E$ is partitioned in two evenly divided parts $(E_1,E_2)$ with $E_1\subseteq R$, $E_2\subseteq B$; that is, such that $||E_1|-|E_2||\le 1$. Note that if $|E|$ is even, then the copy of $G$ has exactly half of its edges in each color class.
	
	Within the range of Ramsey Theory and Extremal Graph Theory, balanceability of graphs deals with the question of the existence and the determination of the minimum number (if it exists) of edges in each color class to guarantee the existence of a balanced copy of a graph $G$ in any 2-coloring of the edges of~$K_n$, as well as with studying the extremal structures. 

	\begin{definition}
		\label{def-balanceable}
		Let $G$ be a simple, finite graph. If there exists an integer $k=k(n)$ such that, for $n$ sufficiently large, every 2-coloring $R \sqcup B$ of the edges of $K_n$ with $|R|,|B|>k$ contains a balanced copy of $G$, then $G$ is \emph{balanceable}. The smallest such $k$ is then called the \emph{balancing number of $G$} and is denoted by $\bal(n,G)$.
	\end{definition}	
	
	In their introductory paper~\cite{CHM19}, Caro, Hansberg and Montejano gave a structural characterization of balanceable graphs. Beyond the computational question of deciding whether a given graph is balanceable or not, there is also the theoretical problem of optimizing the function $k(n)$; that is, providing exact values or bounds for the Tur\'an-type parameter $\bal(n,G)$. 
	
	Several authors studied the problem of balanceability and the optimization problem of determining the balancing number. Caro, Hansberg and Montejano proved in~\cite{CHM19-2} that the only nontrivial balanceable complete graph  with an even number of edges is $K_4$. They also determined the balancing number of $K_4$, as well as of paths and stars, in~\cite{CHM19}. Later, Caro, Lauri and Zarb~\cite{CLZ19} exhaustively studied the balancing number of graphs of at most four edges. Finally, the first, third and fourth authors studied the balanceability of several graph classes~\cite{DHV20}. In \cite{BHMM}, colorings with arbitrary many colors are studied and the corresponding $3$-balancing number for paths is determined upon a constant factor. In this paper, we tackle the case of balanceable cycles (see Section~\ref{ssec-bal-cycles}). 
	
	Although the question of the balancing number is still open for many graph classes, we are also interested in gauging how we may obtain balanceable copies of a nonbalanceable graph, under the relaxation of the 2-coloring we consider. In this paper, we extend the notion of balancing number by extending the class of colorings under consideration to list edge colorings. In this case, each edge receives a nonempty list of colors, and we may choose one among them as needed in order to construct a balanced copy of a graph~$G$. 
	
	More formally, a 2-list coloring of the edges of $K_n$ is a function $L:E(K_n) \rightarrow \{\{r\},\{b\},\{r,b\}\}$, that induces two sets $R$ and $B$, called its \emph{color classes}, which are defined as follows: $R = \{e \in E(K_n)~|~r \in L(e)\}$ and $B = \{e \in E(K_n)~|~b \in L(e)\}$. As we can see, $R \cup B = E(K_n)$, but the two color classes do not necessarily form a partition of the edges of $K_n$. 
	The edges in $R \cap B$ are called \emph{bicolored edges}, in the sense that we can choose their color as needed when looking for a balanced copy of a graph. This leads to a new definition of a balanced copy of a graph, which is a generalization of the previous one:
	\begin{definition}
	    \label{def-balancedCopy}
	    Let $L:E(K_n) \rightarrow \{\{r\},\{b\},\{r,b\}\}$ be a 2-list coloring of the edges of $K_n$ inducing color classes $R$ and $B$. For a given graph $G(V,E)$, a balanced copy of $G$ is a copy of $G$ whose edge-set has a partition $E=E_1\sqcup E_2$ such that $E_1\subseteq R$, $E_2\subseteq B$ and $||E_1|-|E_2||\le 1$.
	\end{definition}
	Note that some of the edges in the copy of $G$ may be in $R\cap B$; for example, if $e\in E_1$ and $e\in R\cap B$ then we say that we \emph{choose} color $r$ for the bicolored edge $e$.
	
	Furthermore, every simple, finite graph $G(V,E)$ has a 2-list edge coloring where we may find a balanced copy of $G$. Namely, the 2-list coloring $L$ of $K_n$ such that $L(e)=\{r,b\}$ for every edge $e$; to see this, observe that in any copy of $G$, we may choose the color $r$ for half its edges and the color $b$ for the rest. This allows us to define the list balancing number of a graph, as follows:
	
	\begin{definition}
		\label{def-balancingNumber}
		Let $G(V,E)$ be a finite, simple graph. For $n \geq |V|$, the \emph{list balancing number of $G$}, denoted by $\lbal(n,G)$, is the smallest integer $k$ such that every 2-list coloring $L$ of the edges of $K_n$ inducing the color classes $R$ and $B$ with $|R|,|B|>k$ contains a balanced copy of $G$.
	\end{definition}
	
	As we can see, the list balancing number is a natural extension of the balancing number. Indeed, every 2-coloring, represented by a partition $E(K_n)=R \sqcup B$, corresponds to a 2-list coloring where every list $L(e)$ has exactly one element. It is important to observe that this extension does not add more complexity to the problem when $\bal(n,G)$ exists, and in which case satisfies $\bal(n,G)<  \frac{1}{2}\binom{n}{2}$ (for details, see \Cref{prop-balIsLbal} in Section~\ref{sec-rbEdgesIntro}). 
	
	The main interest of the list balancing number is the study of non-balanceable graphs, where we interpret having $\lbal(n,G)$ close to $\frac{1}{2}\binom{n}{2}$ as $G$ being close to balanceable, in the sense that a few more than $\frac{1}{2}\binom{n}{2}$ edges from each color (implying that there must be a few bicolored edges) are sufficient to guarantee a balanced copy of $G$. We refer to the needed number of edges that exceeds $\frac{1}{2}\binom{n}{2}$ in each color as the \emph{list-color excess} of edges in each color. For example, we prove in Section~\ref{sec-balancingC4l2}, that, for cycles $C_{4k+2}$ of length congruent to $2$ modulo $4$, which are not balanceable, a list-color  excess of just~$1$ edge in each color is sufficient for guaranteeing a balanced copy of $C_{4k+2}$. In other words, we show that $\lbal(n,C_{4k+2})=\frac{1}{2}\binom{n}{2}$, which is the smallest possible value for the list balancing number of a non-balanceable graph. On the other hand, there are graphs for which a much larger list-color excess in each color is necessary to guarantee a balanced copy of them. Such is the case of $K_5$ where the list-color excess is of order $\theta(n^{\frac{3}{2}})$; see Theorem~\ref{thm-balancingK5}.
	
    One of the key elements of Theorems~\ref{lem-structure-excluded}, \ref{thm-balC4k}, and~\ref{thm-balancingK5} is the use of the \emph{extremal number} (or \emph{Tur\'an number}) of a graph family; recall that, given a graph family $\mathcal{H}$ and an integer $n$, the number $\ex(n,\mathcal{H})$ is the highest number of edges in a graph of order $n$ that does not contain any (not necessarily induced) subgraph in~$\mathcal{H}$.
	
	This paper is organized as follows: we first give, in \Cref{sec-rbEdgesIntro}, some general results about the list-balancing number and about its connection to the balancing number. We then move on to studying the balancing number of the balanceable cycles, that is, odd cycles and cycles on $4k$ vertices, for $k \ge 1$, as well as the list balancing number of non-balanceable cycles, i.e. cycles on $4k+2$ vertices, where $k \ge 2$. We close in \Cref{sec-balancingK5} with the analysis of the complete graph $K_5$, giving asymptotically tight bounds for its list balancing number. 
	
	Note that, in all figures throughout the paper, edges in the color class $R$ will be depicted in red, and edges in the color class $B$ will be depicted in blue.

	\section{List balancing number: preliminaries and general results}
	\label{sec-rbEdgesIntro}

	In this section, we provide general bounds for the list balancing number, which are relevant for graphs where the balancing number does not exist. In particular, \Cref{prop-havingSomeBicoloredEdgesGuaranteesG} says that looking at the number of bicolored edges suffices to get upper bounds; which, in turn, leads to consider the Turán number for a particular class of graphs described in Theorem~\ref{lem-structure-excluded}.
	Before proceeding to these results, we provide a proof of the fact mentioned in the introduction: if $\bal(n,G)$ exists for a graph $G$, then $\lbal(n,G)=\bal(n,G)$. 
		\begin{prop}
	    \label{prop-balIsLbal}
	    Let $G$ be a graph. If $\bal(n,G)$ exists, then $\bal(n,G)=\lbal(n,G)<\frac{1}{2}\binom{n}{2}$.
	\end{prop}

	
	
	
	\begin{proof}
	Consider a graph $G = G(V,E)$ and $n$ for which $\bal(n,G)$ exists; in particular $n\ge |V|$. As we previously discussed, the class considered for $\lbal(n,G)$ extends the class of 2-colorings; this implies that $\bal(n,G)\le \lbal(n,G)$. On the other hand, it is clear by definition that $\bal(n,G)<\frac{1}{2}\binom{n}{2}$.
	
	To establish the equality we prove that every 2-list coloring $E(K_n)=R\cup B$ satisfying $|R|,|B|>\bal(n,G)$ has a balanced copy of $G$. Suppose that there is a 2-coloring $E(K_n)=R'\cup B'$ with $R'\subset R$, $B'\subset B$ and $|R'|,|B'|>\bal(n,G)$, then a balanced copy of $G$ under the 2-coloring $R'\cup B'$ corresponds also to a balanced copy of $G$ under the 2-list coloring $R\cup B$. Hence, it remains to show that we may construct a coloring $R'\cup B'$ with such properties.
	If $|R\setminus B|>\bal(n,G)$ we simply let $R'=R\setminus B$ and $B'=B$. Otherwise, let $R'\subset R$ be an arbitrary subset such that $R\setminus R'\subset B$ and $|R'|=\bal(n,G)+1$, then let $B'=B\setminus R'$. In either case we have $R'\subset R$ and $B'\subset B$. The constraint on the size of $R'$ and $B'$ is clearly satisfied also; for $B'$ in the latter case, observe that $\bal(n,G)< \frac{1}{2}\binom{n}{2}$ implies $|B'|=\binom{n}{2}-|R'|\ge \bal(n,G)$ completing the proof that $\bal(n,G)=\lbal(n,G)$.
	Hence, every 2-list coloring $E(K_n)=R\cup B$ satisfying $|R|,|B|>\bal(n,G)$ has a balanced copy of $G$. 
	\end{proof}
	
	    
	
	\Cref{prop-balIsLbal} immediately implies the following statement:

	\begin{corollary}
	    Let $G = G(V,E)$ be a graph and $n \geq |V|$ be an integer. If $\bal(n,G)$ does not exist, then $\frac{1}{2}\binom{n}{2} \leq \lbal(n,G) < \binom{n}{2}$. 
	\end{corollary}
	
		

	The next result, which is the key for the main theorem of the section, uses a simple relation between the sizes of the color classes and the necessary number of bicolored edges. 
	
	\begin{prop}
		\label{prop-havingSomeBicoloredEdgesGuaranteesG}
		Let $G$ be a graph and let $b$ be a positive integer. If every 2-list edge coloring with at least $b$ bicolored edges has a balanced copy of $G$ then  $\lbal(n,G) \leq \frac{1}{2}\binom{n}{2}+\left\lceil\frac{b}{2}\right\rceil-1$.
	\end{prop}
	
	\begin{proof}
	First observe that an inclusion-exclusion argument gives that if $R\cup B$ are the color classes of a 2-list edge coloring of $K_n$ satisfying $|R|=|B| = \frac{1}{2}\binom{n}{2}+m$, then there are exactly $2m$ bicolored edges. 
	Consequently, for any  2-list edge coloring inducing color classes $R$ and $B$ with $|R|,|B| \geq \frac{1}{2}\binom{n}{2}+\left\lceil\frac{b}{2}\right\rceil$, there are at least $b$ bicolored edges, and so, by hypothesis, there is a balanced copy of $G$.
	\end{proof}
	
	
	
	The following theorem uses the idea of exploiting the flexibility of bicolored edges. Once we have a copy of $G$ we have to choose the color of the bicolored edges, and we do so according to the edges which have a unique color. In particular, if more than half the edges of such copy are bicolored then we may distribute these between the two color classes to balance the copy, regardless of the color of the rest of the edges. With this perspective, a general bound on the list balancing number may be reduced to guaranteeing that the 2-list edge colorings have enough bicolored edges. To this aim, we need to consider, for a graph $G$, the family of all its subgraphs having half of the edges, 
	
	$$\half(G)  =  \left\{   H :  H \le G,  e(H) = \left\lfloor\frac{e(G)}{2} \right\rfloor, H \mbox{ has no isolates} \right\}.$$
	
	We note at this point that this family was already used in \cite{CHLZ20} to gain a similar insight in studying the existence of balanced copies of spanning subgraphs of a $2$-colored $K_n$.

	\begin{theorem}
		\label{lem-structure-excluded}
		Let $G = G(V,E)$ be a graph and $n \geq |V|$ be an integer. Then we have 
		$$\lbal(n,G) \leq \frac{1}{2}\binom{n}{2}+\left\lceil \frac{\ex(n,\half(G))}{2} \right\rceil.$$
	\end{theorem}
	
	\begin{proof}
		Let $G=G(V,E)$ be a graph, and $\mathcal{H}$ be the family of subgraphs of $G$ with at least $\frac{|E|}{2}$ edges. Now, let $R$ and $B$ be the color classes induced by a 2-list edge coloring of $K_n$ with $|R|,|B|>\frac{1}{2}\binom{n}{2}+\lceil \frac{\ex(n,\half(G))}{2} \rceil$. This implies that there are at least $\ex(n,\half(G))+1$ bicolored edges. In particular, since $K_n$ is of order $n$, the subgraph of $K_n$ induced by the bicolored edges contains a graph in $\half(G)$, say $H$. Starting from $H$, we can complete with other edges to construct a copy of $G$. This copy has at least half its edges that are bicolored, and thus we can make it balanced. Hence, we can find a balanced copy of $G$, proving the upper bound on $\lbal(n,G)$. 
	\end{proof}

	\Cref{lem-structure-excluded} gives a general upper bound on the list balancing number of graphs which, by \Cref{prop-balIsLbal}, is only relevant when the balancing number does not exist. This fairly general bound may be tight (up to the order of the second term $\ex(n,\half(G))$), as is the case of $K_5$. In \Cref{sec-balancingK5}, we will use this result to give an upper bound on the list balancing number for $K_5$, which will be matched with a lower bound of the same order. However this general upper bound can also be far from the exact value of the list balancing number, as can be noticed already in Section in \Cref{sec-balancingC4l2}, where we determine the list balancing number for the unbalanceable cycles.
	
	
	In the remainder of the paper we say that a 2-list coloring has a \emph{list-color excess} of $b$ edges, referring to the number of edges in each color by which $\frac{1}{2}\binom{n}{2}$ is at least surpassed. More precisely, we say that a 2-list coloring $E(K_n)=R\cup B$ has a list-color excess of $b$ edges, if $|R|,|B|\ge \frac{1}{2}\binom{n}{2}+b$. In such a case, since $|R \cap B| = |R| + |B| - |R \cup B| = |R| + |B| - \binom{n}{2} \ge 2b$, it clearly follows that there have to be at least $2b$ bicolored edges.


	\section{The balancing number and list balancing number of cycles}\label{sec-cycles}
	
    In \cite{DHV20}, it was observed that the cycle $C_{4k}$ is balanceable while the cycle $C_{4k+2}$ is not. 
    We note here that all odd cycles are also balanceable, and provide exact values for their balancing number. Moreover, we give tight bounds, up to the first order term, for the balancing number for cycles of length $4k$, for $k \ge 1$. Finally, we determine the list balancing number of $C_{4k+2}$, for $k \ge 1$.  
    
    \subsection{Balanceable cycles}\label{ssec-bal-cycles}
	
	Theorem~\ref{thm-balancingOddCycles} below, which deals with odd-length cycles is a direct consequence of the balanceability of paths of even length. We denote with $P_\ell$ the path on $\ell$ edges (and thus $\ell+1$ vertices); the exact values of $\bal(n,P_{\ell})$ are obtained in~\cite[Theorem~3.7]{CHM19}. 

	\begin{theorem}
		\label{thm-balancingOddCycles}
		Let $k$ be a positive integer, let $n \ge \frac{9}{2}k^2 +\frac{13}{4}k+\frac{49}{32}$, and let $\alpha \in \{-1,1\}$. We have the following:
		$$ \lbal(n,C_{4k+\alpha}) = \bal(n,C_{4k+\alpha}) = \bal(n,P_{4k+\alpha -1}) = (k-1)n-\frac{1}{2} (k^2-k -1-\alpha).$$
	\end{theorem}
	
	\begin{proof}
	    Let $k$ be a positive integer. The equality $\lbal(n,C_{4k+\alpha}) = \bal(n,C_{4k+\alpha})$ is clear from \Cref{prop-balIsLbal}. Now we will prove that the balancing number of odd cycles is equal to the balancing number of paths with one edge less. We will demonstrate the result only for $C_{4k+1}$ (i.e for $\alpha = 1$) since the exact same arguments can be made for $C_{4k-1}$.
		
		First, we prove that $\bal(n,C_{4k+1}) \leq \bal(n,P_{4k})$. Assume that we have a 2-coloring $R \sqcup B$ of the edges of $K_n$ with $|R|,|B| > \bal(n,P_{4k})$. This implies that there is a balanced copy of $P_{4k}$; 
		that is, a path with equal number of edges in $R$ and $B$. Regardless of the color of the edge connecting the endpoints of the path, the addition of this edge to the path creates a balanced cycle. Hence, we obtain the claimed upper bound on $\bal(n,C_{4k+1})$.
		
		Now, we prove that $\bal(n,P_{4k})\le \bal(n,C_{4k+1})$. Assume that we have a 2-coloring $R \sqcup B$ of the edges of $K_n$ with $|R|,|B| > \bal(n,C_{4k+1})$. This implies that there is a balanced copy of $C_{4k+1}$; without loss of generality, assume that this cycle contains $2k$ red edges and $2k+1$ blue edges. The path obtained from the cycle by deleting one of the blue edges is a balanced path of length $4k$, completing the proof that $\bal(n,C_{4k+1}) = \bal(n,P_{4k})$.
		\end{proof}
	
	The problem of finding the exact value of the balancing number of cycles of length $4k$ is more challenging; Theorem~\ref{thm-balC4k} below gives an upper and a lower bound for $\bal(n,C_{4k})$ which are tight up to the first term, $(k-1)n$; note that, contrary to the case of odd-length cycles, we need additional edges in each color class (of the order of $k^2$) to guarantee a balanced copy of $C_{4k}$. 
	
	\begin{theorem}
		\label{thm-balC4k}
		Let $k$ be a positive integer. For $n\ge \frac{9}{2}k^2 +\frac{13}{4}k+\frac{49}{32}$, we have the following:
		$$ (k-1)n-(k-1)^2 \leq  \lbal(n,C_{4k}) = \bal(n,C_{4k}) < (k-1)n +12k^2+3k. $$
	\end{theorem}
	
	The equality $\lbal(n,C_{4k}) = \bal(n,C_{4k})$ is clear from \Cref{prop-balIsLbal}. Thus, the next two lemmas directly prove the bounds of Theorem~\ref{thm-balC4k}. First, we show that there is a natural 2-coloring avoiding any balanced cycle of length $4k$ which provides us with a lower bound for $\bal(n,C_{4k})$.
	
	\begin{lemma}
	\label{lem-balC4k-lowerBound}
	For integers $n\ge 4k$, we have $\bal(n,C_{4k}) \geq (k-1)n-(k-1)^2$.
\end{lemma}

\begin{proof}
We will give a $2$-coloring $R \sqcup B$ of the edges of $K_{n}$ with $|B|\ge |R|=(k-1)n-(k-1)^2$ and without a balanced copy of $C_{4k}$. To this aim, let $V(K_{n})=V_{1} \sqcup V_{2}$, where $|V_{1}|=k-1$ and $|V_{2}|=n-k+1$, and we color the edges in $E(V_1,V_2)$ with red, and the remaining edges get the color blue. This coloring satisfies $|R|=(k-1)(n-k+1) = (k-1)n-(k-1)^2$, and it is not difficult to verify that $|B|\ge |R|$ for any $k,n\ge 1$. Furthermore, any $4k$-cycle in this coloring can have at most $k-1$ vertices in $V_1$ and thus have at most $2k-2$ red edges. It follows that we cannot get a balanced copy of $C_{4k}$, which implies that $\bal(n,C_{4k}) \geq (k-1)(n-k+1)$.
\end{proof}
	
	Similar to the proof idea of Theorem~\ref{thm-balancingOddCycles}, an upper bound for $\bal(n,C_{4k})$ can be given by constructing a balanced copy of $C_{4k}$ from a balanced copy of $C_{4k-1}$. We will show that, if this construction is not possible, then certain structure for the 2-coloring of the edges of $K_n$ is forced, in which we are able to find, in turn, a balanced $4k$-cycle by means of a long red path that is glued together with a long blue path, together with some extra edges that close the cycle. To guarantee the existence of the long red path, we make use of the extremal number for paths. Note that the additional edges (namely $(k-1)n+11k^2+3k-\bal(n, C_{4k-1})$) are necessary to guarantee this extremal number is exceeded. 
	
	\begin{lemma}
		\label{lem-balC4k-upperBound}
		For $n\ge \frac{9}{2}k^2 +\frac{13}{4}k+\frac{49}{32}$, we have $\bal(n,C_{4k}) \le (k-1)n+12k^2+3k$.
	\end{lemma}
	
	\begin{proof}
	For $k\ge 1$ we have that $\frac{9}{2}k^2 +\frac{13}{4}k+\frac{49}{32}\ge  10k-2$; thus we may assume that we can apply \Cref{thm-balancingOddCycles} and that $n\ge 10k-2$, which is a sufficient assumption on $n$ for all forthcoming arguments.
	
	We first verify that the condition $\min\{|R|,|B|\}\ge (k-1)n+11k^2+3k$ is satisfiable; this is $\binom{n}{2}\ge 2(k-1)n+22k^2+6k$. Using that $k\ge 1$, it suffices to verify that $n(n-4k+3)\ge 56k^2$. Since $n\ge 10k-2$ we have, indeed, 
	\begin{align*}
	    n(n-4k+3)\ge (10k-2)(6k+1)=56k^2+(2k-2)(2k+1)\ge 56k^2;
	\end{align*}
	so we may consider any 2-coloring of the edges of $K_n$ with $|R|,|B|  > (k-1)n+12k^2+3k$.
	
	We now prove the lemma by contradiction. Let $R \sqcup B$ be a 2-coloring of the edges of $K_n$ with $|R|,|B| \ge (k-1)n+11k^2+3k$. Assume that this coloring has no copy of a balanced $4k$-cycle. By Theorem~\ref{thm-balancingOddCycles} there is a balanced copy $C$ of $C_{4k-1}$ in this coloring. Without loss of generality, we may assume that $C$ consists of a cycle with $2k$ blue edges and $2k-1$ red edges. This implies that $C$ has, at some place, a red edge followed by two blue edges: say $C$ has consecutive vertices $u_0,u_1,u_2,u_3$ where $u_0u_1 \in R$ and $u_1u_2,u_2u_3 \in B$. 
	
	Let $V = V(K_n)$ and $W = V(C)$. 
	In what follows we will infer a structure among the vertices in $V\setminus W$ which will lead to a contradiction to the initial assumption that there is no balanced $4k$-cycle.  
	The next three claims stem from the fact that some specific structure outside of $C$ would give a balanced $4k$-cycle. Let $X$ (resp. $Y$) correspond to the sets of vertices $v\in V\setminus W$ such that $u_1v\in R$ (resp. $u_1v\in B$). Note that $V\setminus W=X\cup Y$, though either $X$ or $Y$ may be empty. We will now strengthen the structure with three claims.
	
%
%
%
	
	\begin{claim}
	    \label{clm-BB}
	    For each $v\in Y$, $u_iv\in B$  for all $1\le i\le 3$.
	\end{claim}
	
	\claimproof{clm-BB}
	Let $v$ be a vertex in $Y$. By definition $u_1v \in B$. If $u_2v\in R$, then we may extend $C$ by replacing the edge $u_1u_2$ with the path $u_1vu_2$ to obtain a balanced $4k$-cycle, a contradiction (see \Cref{fig-clmBB}). It follows that $u_2v\in B$. Now, applying a similar argument (replacing $u_2u_3$ with the path $u_2vu_3$), we can conclude that $u_3v \in B$.
	\claimqed
	
	\begin{claim}
	    \label{clm-X}
	    For each $v\in X$, $u_0v\in B$ and $u_2v\in R$.
	\end{claim}
	
	\claimproof{clm-X}
	Let $v$ be a vertex in $X$. By definition, $u_1v\in R$. The same argument than in the proof of Claim~\ref{clm-BB} gives that $u_2v\in R$. Now, assume by contradiction that $u_0v\in R$. Then, we may extend $C$ by replacing the edge $u_0u_1$ with the path $u_0vu_1$, and thus obtain a balanced $4k$-cycle, a contradiction (see \Cref{fig-clmX}).
	\claimqed
	
	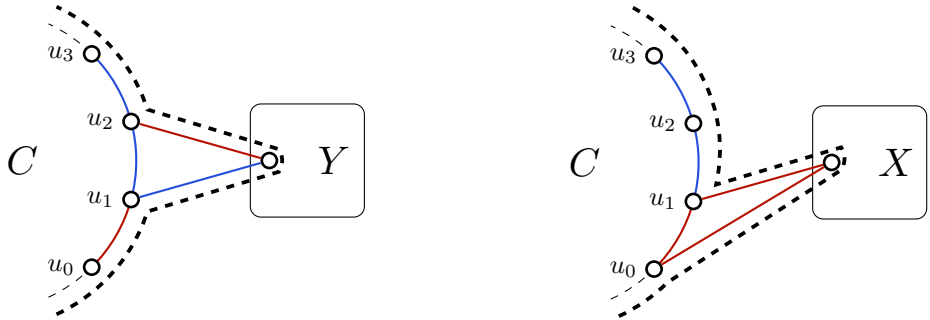
\begin{figure}[h]
	    \centering
	    \begin{subfigure}{.45\textwidth}
	        \centering
	        \begin{tikzpicture}
    	        \draw[dashed] (-65:2) arc (-65:-45:2);
    	        \draw[thick,rouge] (-45:2) arc (-45:-15:2);
    	        \draw[thick,bleu] (-15:2) arc (-15:45:2);
    	        \draw[dashed] (45:2) arc (45:65:2);
    	        \node[noeud] (u0) at (-45:2) {};
    	        \node[noeud] (u1) at (-15:2) {};
    	        \node[noeud] (u2) at (15:2) {};
    	        \node[noeud] (u3) at (45:2) {};
    	        \draw (u0) node[left=1mm] {$u_0$};
    	        \draw (u1) node[left=1mm] {$u_1$};
    	        \draw (u2) node[left=1mm] {$u_2$};
    	        \draw (u3) node[left=1mm] {$u_3$};
    	        \draw (0.5,0) node[scale=1.5] {$C$};
    	        
    	        \draw[rounded corners] (3.5,-0.75) rectangle (5,0.75);
    	        \draw (4.6,0) node[scale=1.5] {$Y$};
    	        
    	        \node[noeud] (v) at (3.75,0) {};
    	        
    	        \draw[thick,bleu] (u1)to(v);
    	        \draw[thick,rouge] (u2)to(v);
    	        
    	        \draw[line width=0.5mm,dashed] (-65:2.25) arc (-65:-17.5:2.25) to (3.875,-0.15) to[bend right] (3.875,0.15) to (17.5:2.25) arc (17.5:65:2.25);
    	    \end{tikzpicture}
	        \caption{The proof of Claim~\ref{clm-BB}: if, for any $v \in Y$, $u_2v \in B$, then we can alter $C$ to construct a balanced $4k$-cycle.}
	        \label{fig-clmBB}
	    \end{subfigure}~~\begin{subfigure}{.45\textwidth}
	        \centering
	        \begin{tikzpicture}
    	        \draw[dashed] (-65:2) arc (-65:-45:2);
    	        \draw[thick,rouge] (-45:2) arc (-45:-15:2);
    	        \draw[thick,bleu] (-15:2) arc (-15:45:2);
    	        \draw[dashed] (45:2) arc (45:65:2);
    	        \node[noeud] (u0) at (-45:2) {};
    	        \node[noeud] (u1) at (-15:2) {};
    	        \node[noeud] (u2) at (15:2) {};
    	        \node[noeud] (u3) at (45:2) {};
    	        \draw (u0) node[left=1mm] {$u_0$};
    	        \draw (u1) node[left=1mm] {$u_1$};
    	        \draw (u2) node[left=1mm] {$u_2$};
    	        \draw (u3) node[left=1mm] {$u_3$};
    	        \draw (0.5,0) node[scale=1.5] {$C$};
    	        
    	        \draw[rounded corners] (3.5,-0.75) rectangle (5,0.75);
    	        \draw (4.6,0) node[scale=1.5] {$X$};
    	        
    	        \node[noeud] (v) at (3.75,0) {};
    	        \draw (v) node[right=1.25mm] {};
    	        
    	        \draw[thick,rouge] (u1)to(v);
    	        \draw[thick,rouge] (u0)to(v);
    	        
    	        \draw[line width=0.5mm,dashed] (-65:2.25) arc (-65:-45:2.25) to (3.875,-0.1) to[bend right] (3.875,0.2) to (-7.5:2.25) arc (-12.5:65:2.125);
    	    \end{tikzpicture}
	        \caption{The proof of Claim~\ref{clm-X}: if, for any $v \in X$, $u_0v \in R$, then we can alter $C$ to construct a balanced $4k$-cycle.}
	        \label{fig-clmX}
	    \end{subfigure}
	    
	    \caption{Strengthening the structure: Claims~\ref{clm-BB} and~\ref{clm-X}.}
	    \label{fig-clmsBBX}
	\end{figure}
	
	We now have a more constrained structure, which is depicted on \Cref{fig-structure1}.
	
	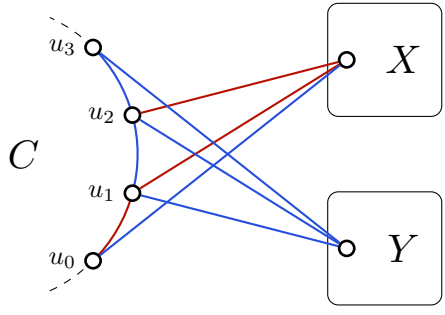
\begin{figure}[h]
    	\centering
    	\begin{tikzpicture}
    		\draw[dashed] (-65:2) arc (-65:-45:2);
    		\draw[thick,rouge] (-45:2) arc (-45:-15:2);
    		\draw[thick,bleu] (-15:2) arc (-15:45:2);
    		\draw[dashed] (45:2) arc (45:65:2);
    		\node[noeud] (u0) at (-45:2) {};
    		\node[noeud] (u1) at (-15:2) {};
    		\node[noeud] (u2) at (15:2) {};
    		\node[noeud] (u3) at (45:2) {};
    		\draw (u0) node[left=1mm] {$u_0$};
    		\draw (u1) node[left=1mm] {$u_1$};
    		\draw (u2) node[left=1mm] {$u_2$};
    		\draw (u3) node[left=1mm] {$u_3$};
    		\draw (0.5,0) node[scale=1.5] {$C$};
    		
    		\draw[rounded corners] (4.5,0.5) rectangle (6,2);
    		\draw (5.5,1.25) node[scale=1.5] {$X$};
    		\draw[rounded corners] (4.5,-2) rectangle (6,-0.5);
    		\draw (5.5,-1.25) node[scale=1.5] {$Y$};
    		
    		\node[noeud] (vx) at (4.75,1.25) {};
    		\draw[thick,rouge] (u1)to(vx);
    		\draw[thick,bleu] (u0)to(vx);
    		\draw[thick,rouge] (u2)to(vx);
    		\node[noeud] (vy) at (4.75,-1.25) {};
    		\draw[thick,bleu] (u1)to(vy);
    		\draw[thick,bleu] (u2)to(vy);
    		\draw[thick,bleu] (u3)to(vy);
    		
    %
    %
    %
    %
    %
    %
    	\end{tikzpicture}
    	\caption{The structure after Claims~\ref{clm-BB} and~\ref{clm-X}. All the edges from the $u_i$s to vertices in $X$ and $Y$ follow this structure.}
    	\label{fig-structure1}
    \end{figure}
	
	\begin{claim}
	    \label{clm-XY}
	    We have $E(X, Y) \subseteq R$, and $E(X) \cup E(Y) \subseteq B$.
	\end{claim}
	
	\claimproof{clm-XY}
	First, assume by contradiction that there are $v,v'\in X$ such that $vv'\in R$. By Claim~\ref{clm-X}, the path $u_0vv'u_2$ consists of two red edges and one blue edge; thus we may extend $C$ by replacing $u_0u_1u_2$ with the path $u_0vv'u_2$ to obtain a balanced $4k$-cycle, a contradiction (see \Cref{fig-clmXY-1}). It follows that $vv'\in B$ for all $v,v'\in X$.
	
	Next, assume by contradiction that there are $v,v'\in Y$ such that $vv'\in R$. By Claim~\ref{clm-BB}, the path $u_1vv'u_3$ consists of two blue edges and one red edge; thus we may extend $C$ by replacing $u_1u_2u_3$ with the path $u_1vv'u_3$ to obtain a balanced $4k$-cycle, a contradiction (see \Cref{fig-clmXY-2}). It follows that $vv'\in B$ for all $v,v'\in Y$.
	
	Finally, assume by contradiction that there are $v\in X$ and $v'\in Y$ such that $vv'\in B$. Since $u_1v \in R$ (by definition) and $u_3v' \in B$ (by Claim~\ref{clm-BB}), we can replace the path $u_1u_2u_3$ by the path $u_1vv'u_3$, and obtain a balanced $4k$-cycle, a contradiction (see \Cref{fig-clmXY-3}). It follows that for each $v\in X$ and $v'\in Y$, $vv'\in R$.
	\claimqed
	
	\begin{figure}[h]
	    \centering
	    
	    \begin{subfigure}{0.3\linewidth}
	        \centering
	        \begin{tikzpicture}
	            \draw[dashed] (-65:2) arc (-65:-45:2);
    	        \draw[thick,rouge] (-45:2) arc (-45:-15:2);
    	        \draw[thick,bleu] (-15:2) arc (-15:45:2);
    	        \draw[dashed] (45:2) arc (45:65:2);
    	        \node[noeud] (u0) at (-45:2) {};
    	        \node[noeud] (u1) at (-15:2) {};
    	        \node[noeud] (u2) at (15:2) {};
    	        \node[noeud] (u3) at (45:2) {};
    	        \draw (u0) node[left=1mm] {$u_0$};
    	        \draw (u1) node[left=1mm] {$u_1$};
    	        \draw (u2) node[left=1mm] {$u_2$};
    	        \draw (u3) node[left=1mm] {$u_3$};
    	        \draw (0.75,0) node[scale=1.5] {$C$};
    	        
    	        \draw[rounded corners] (3,0.5) rectangle (4.5,2);
    	        \draw (4.1,1.25) node[scale=1.5] {$Y$};
    	        \draw[rounded corners] (3,-2) rectangle (4.5,-0.5);
    	        \draw (4.1,-1.25) node[scale=1.5] {$X$};
    	        
    	        \node[noeud] (v) at (3.25,-1.625) {};
    	        \node[noeud] (vp) at (3.25,-0.875) {};
    	        
    	        \draw[thick,bleu] (u0)to(v);
    	        \draw[thick,rouge] (v)to(vp)to(u2);
    	        
    	        \draw[line width=0.5mm,dashed] (-65:2.25) arc (-65:-45:2.25) to (3.5,-1.8) to (3.5,-0.875) to (17.5:2.25) arc (17.5:65:2.125);
	        \end{tikzpicture}
	        \caption{There are no red edges in $X$.}
	        \label{fig-clmXY-1}
	    \end{subfigure}~~~~~\begin{subfigure}{0.3\linewidth}
	        \centering
	        \begin{tikzpicture}
	            \draw[dashed] (-65:2) arc (-65:-45:2);
    	        \draw[thick,rouge] (-45:2) arc (-45:-15:2);
    	        \draw[thick,bleu] (-15:2) arc (-15:45:2);
    	        \draw[dashed] (45:2) arc (45:65:2);
    	        \node[noeud] (u0) at (-45:2) {};
    	        \node[noeud] (u1) at (-15:2) {};
    	        \node[noeud] (u2) at (15:2) {};
    	        \node[noeud] (u3) at (45:2) {};
    	        \draw (u0) node[left=1mm] {$u_0$};
    	        \draw (u1) node[left=1mm] {$u_1$};
    	        \draw (u2) node[left=1mm] {$u_2$};
    	        \draw (u3) node[left=1mm] {$u_3$};
    	        \draw (0.75,0) node[scale=1.5] {$C$};
    	        
    	        \draw[rounded corners] (3,0.5) rectangle (4.5,2);
    	        \draw (4.1,1.25) node[scale=1.5] {$Y$};
    	        \draw[rounded corners] (3,-2) rectangle (4.5,-0.5);
    	        \draw (4.1,-1.25) node[scale=1.5] {$X$};
    	        
    	        \node[noeud] (v) at (3.25,0.875) {};
    	        \node[noeud] (vp) at (3.25,1.625) {};
    	        
    	        \draw[thick,bleu] (u1)to(v);
    	        \draw[thick,rouge] (v)to(vp);
    	        \draw[thick,bleu] (u3)to(vp);
    	        
    	        \draw[line width=0.5mm,dashed] (-65:2.25) arc (-65:-15:2.25) to (3.5,0.875) to (3.5,1.8) to (47.5:2.25) arc (47.5:65:2.125);
	        \end{tikzpicture}
	        \caption{There are no red edges in $Y$.}
	        \label{fig-clmXY-2}
	    \end{subfigure}~~~~~\begin{subfigure}{0.3\linewidth}
	        \centering
	        \begin{tikzpicture}
	            \draw[dashed] (-65:2) arc (-65:-45:2);
    	        \draw[thick,rouge] (-45:2) arc (-45:-15:2);
    	        \draw[thick,bleu] (-15:2) arc (-15:45:2);
    	        \draw[dashed] (45:2) arc (45:65:2);
    	        \node[noeud] (u0) at (-45:2) {};
    	        \node[noeud] (u1) at (-15:2) {};
    	        \node[noeud] (u2) at (15:2) {};
    	        \node[noeud] (u3) at (45:2) {};
    	        \draw (u0) node[left=1mm] {$u_0$};
    	        \draw (u1) node[left=1mm] {$u_1$};
    	        \draw (u2) node[left=1mm] {$u_2$};
    	        \draw (u3) node[left=1mm] {$u_3$};
    	        \draw (0.75,0) node[scale=1.5] {$C$};
    	        
    	        \draw[rounded corners] (3,0.5) rectangle (4.5,2);
    	        \draw (4.1,1.25) node[scale=1.5] {$Y$};
    	        \draw[rounded corners] (3,-2) rectangle (4.5,-0.5);
    	        \draw (4.1,-1.25) node[scale=1.5] {$X$};
    	        
    	        \node[noeud] (v) at (3.25,-1.25) {};
    	        \node[noeud] (vp) at (3.25,1.25) {};
    	        
    	        \draw[thick,rouge] (u1)to(v);
    	        \draw[thick,bleu] (u3)to(vp)to(v);
    	        
    	        \draw[line width=0.5mm,dashed] (-65:2.25) arc (-65:-22.5:2.25) to (3.5,-1.625) to (3.5,1.5) to (47.5:2.25) arc (47.5:65:2.125);
	        \end{tikzpicture}
	        \caption{There are no blue edges between $X$ and $Y$.}
	        \label{fig-clmXY-3}
	    \end{subfigure}
	    
	    \caption{Strengthening the structure: Claim~\ref{clm-XY}. In every case, we can use $C$ to get a balanced $4k$-cycle, a contradiction.}
	    \label{fig-clmXY}
	\end{figure}
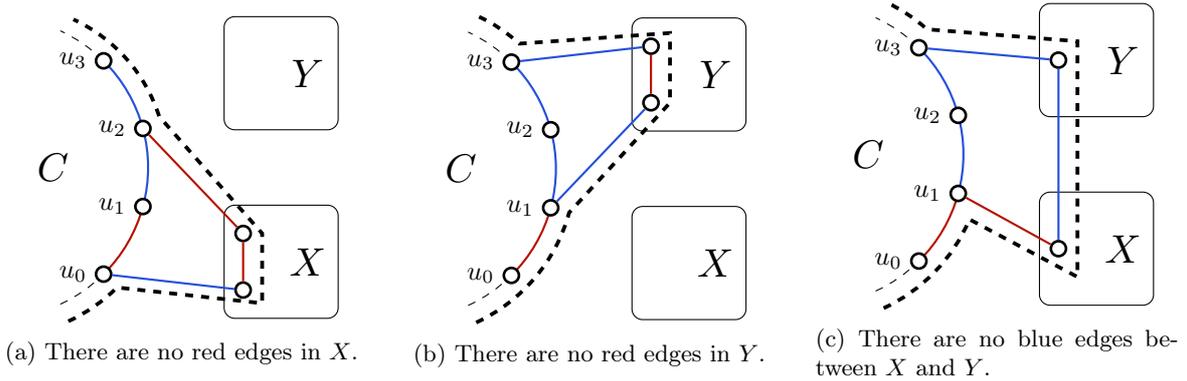

    We now use the structure we found in \Cref{clm-XY} to find a contradiction. Recall that $n\ge 10k-2$ which implies that $|X\cup Y|=|V \setminus W|\ge n-4k+1\ge 6k-1$ and so $\max\{|X|,|Y|\}\ge 3k$.
    
    For the remainder of the proof, we have two possibilities: either $|X| \leq |Y|$ or $|X|>|Y|$. However, note that those two cases are symmetrical since we will not care about the specific colors of the edges between the vertices $u_0, u_1, u_2, u_3$ and $X \cup Y$, but rather more in general within and between the sets $W, X$, and $Y$.
    
    Hence, without loss of generality, we assume that $|X|\le|Y|$. This condition will imply $|X|<k$. Indeed, assume by contradiction that $|X| \geq k$. Then, we can obtain a balanced $4k$-cycle by taking a blue path of length $2k$ within $Y$, and then a red path of length $2k$ closing the cycle by going back and forth $2k$ times between $X$ and $Y$ (which is possible since $|Y|\ge 3k$ and $|X|\ge k$). This contradiction implies that $|X| < k$.

	Thus, we have a partition $V =Y \sqcup (W\cup X)$ where all the edges within $Y$ are blue and $|Y|\ge n-5k+2$ (since $|X\cup Y|= n-4k+1$ and $|X|\le k-1$). We will now study the number of red edges in $E(Y, W \cup X)$. Let $H$ be the bipartite graph induced by the set of red edges contained in $E(Y, W \cup X)$. We start by giving a lower bound on the number of edges in $H$, which is the number of red edges in $K_n$ minus the number of red edges in $E(W \cup X)$; recall that $E(X) \cup E(Y) \subseteq B$ by \Cref{clm-XY}. Hence, we have:
    \begin{align*}
    |E(W\cup X)\cap R|&\le e(W)+e(W,X)< \binom{4k-1}{2} + k(4k-1) = (4k-1)(3k-1)   \end{align*}
    and so, we have:
    \begin{align*}
    e(H)=|R|-|E(W\cup X)\cap R|&\ge (k-1)n+12k^2+3k-(4k-1)(3k-1)\\ 
    &\geq (k-1)n +10k-1. 
    \end{align*}
    However, Theorem 5.5 in \cite{FS13} states that $\ex(n,P_{2k-1})\le (k-1)n$, which means that, in a graph with $n$ vertices and at least $(k-1)n$ edges, there is a path on $2k-1$ edges. As a consequence, there is a path $P$ of length $2k-1$ edges in $H$. Since $P$ has an even number of vertices, we may assume that $P=v_1w_1v_2\ldots w_{k-1}v_kw_k$ with all $v_i\in Y$ and all $w_i\in W\cup X$.
    
    Let $H'=(Y',X')$ be the subgraph of $H$ induced by $Y'=Y\setminus \{v_1,\ldots, v_k\}$ and $X'=(W\cup X)\setminus \{w_1,\ldots,w_{k-1}\}$. Observe that $|X'|=|W\cup X|-(k-1)\le 4k-1$ and using the lower bound on $e(H)$, we get 
    \begin{align*}
        e(H')&=e(H)-e(Y,\{w_1,\ldots,w_{k-1}\})-e(W \cup X,\{v_1,\ldots,v_k\})\\
        &\ge e(H)-(n-5k+2)(k-1)-(4k-1)k\\
        &= e(H)-(k-1)n+k^2-6k+2\\
        &\ge (k-1)n+10k-1-(k-1)n+k^2-6k+2\\
        &>4k. 
    \end{align*}
    It follows that $e(H')> |X'|$ and, by the pigeonhole principle, there is a vertex $w\in X'$ that has two neighbors $v$ and $v'$ in $Y'$.

     However, this allows us to construct a balanced $4k$-cycle. Indeed, start from $v_1$ and take the path $P$ all the way to $v_k$ (this gives us $2k-2$ red edges), then go to $v$ and take the path $vwv'$ (this gives one blue edge and two red edges), 
    and finally take a path of $2k-1$ blue edges that ends back in $v_1$ (by using vertices $x_1,\ldots,x_{2k-2}$ in $Y'$). Note that we can deliberately select the $x_i$'s  to be distinct from $v$ and $v'$, since $|Y'| \geq 2k$. This cycle, shown on \ref{fig-finalContradiction}, has $2k$ edges in each color class, thus we have a contradiction.
    
    \begin{figure}[h]
        \centering
        \begin{tikzpicture}
            \draw[rounded corners,thick] (0,0) rectangle (3,5);
    		\draw (1.5,5.5) node[scale=1.5] {$W \cup X$};
    		\draw[rounded corners,thick] (4,0) rectangle (7,5);
    		\draw (5.5,5.5) node[scale=1.5] {$Y$};
    		\draw[thick,dashed] (0,2.5)to(3,2.5);
    		\draw (0.5,1.25) node[scale=1.5] {$X'$};
    		\draw[thick,dashed] (4,2.5)to(7,2.5);
    		\draw (6.5,1.25) node[scale=1.5] {$Y'$};
    		
    		\node[noeud] (w1) at (1.5,4.5) {};
    		\draw (w1) node[left=1mm] {$w_1$};
    		\draw (1.5,4.125) node {$\vdots$};
    		\node[noeud] (wk1) at (1.5,3.5) {};
    		\draw (wk1) node[left=1mm] {$w_{k-1}$};
    		\node[noeud] (wk) at (1.5,3) {};
    		\draw (wk) node[left=1mm] {$w_k$};
    		
    		\node[noeud] (v1) at (5,4.5) {};
    		\draw (v1) node[right=1mm] {$v_1$};
    		\draw (5,3.875) node {$\vdots$};
    		\node[noeud] (vk) at (5,3) {};
    		\draw (vk) node[right=1mm] {$v_k$};
    		
    		\node[noeud] (v) at (5,2) {};
    		\draw (v) node[right=1mm] {$v$};
    		\node[noeud] (vp) at (5,0.5) {};
    		\draw (vp) node[right=1mm] {$v'$};
    		\node[noeud] (w) at (1.5,1.25) {};
    		\draw (w) node[left=1mm] {$w$};
    		
    		\node[noeud] (x1) at (6,0.5) {};
    		\draw (x1) node[right=1mm] {$x_1$};
    		\node[noeud] (x2k2) at (6,2) {};
    		\draw (x2k2) node[right=1mm] {$x_{\ell}$};
    		
    		\draw[thick,rouge] (v1)to(w1);
    		\draw[thick,rouge,dashed] (w1)to(2.5,4.25);
    		\draw[thick,rouge,dashed] (wk1)to(2.5,3.75);
    		\draw[thick,rouge] (wk1)to(vk);
    		\draw[thick,bleu] (vk)to(v);
    		\draw[thick,rouge] (v)to(w)to(vp);
    		\draw[thick,bleu,out=270,in=270] (vp)to(x1);
    		\draw[thick,bleu,dashed] (x1)to(6,1);
    		\draw[thick,bleu,dashed] (6,1.5)to(x2k2);
    		\draw[thick,bleu,bend right] (x2k2)to(v1);
        \end{tikzpicture}
        \caption{Constructing a balanced $4k$-cycle by using the structure between $W \cup X$ and $Y$ (we have $\ell = 2k-2$).}
        \label{fig-finalContradiction}
    \end{figure}
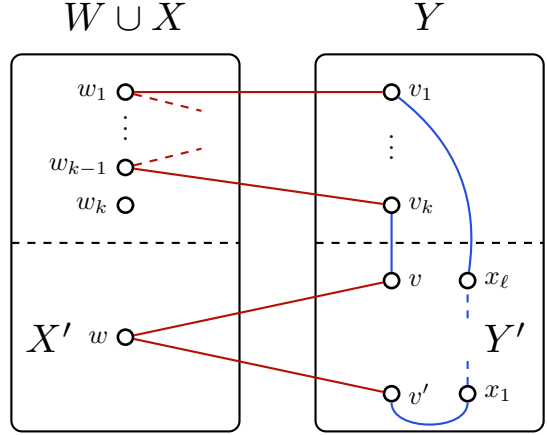
    
    This contradiction proves the lemma.
    \end{proof}

\subsection{Non-balanceable cycles}
	\label{sec-balancingC4l2}
	
	
	In this section, we obtain the exact value of the list balancing number for $C_{4k+2}$, for $k \ge 1$, which represents the class of non-balanceable cycles. This case is remarkable because it suffices that each color class covers at least half the edges in $K_n$ plus one additional edge (i.e. the coloring has a list-color excess of $1$), which implies that, no matter how is the coloring, there are necessarily at least two bicolored edges. However, the construction of the balanced cycle in \Cref{thm-balancingC4l2} uses the existence of only one bicolored edge, justifying the heuristic that the list balancing number (when it is at least $\frac{1}{2}\binom{n}{2}$) provides a measure of how close the graph is to being balanceable. 
	
	First, let us apply \Cref{lem-structure-excluded} to get a first upper bound for the list balancing number of non-balanceable cycles, which we will then prove to be far from the exact value of the parameter.
	
	\begin{corollary}
		\label{cor-balancingC4l2}
		Let $k$ be a positive integer and $n$ be sufficiently large. Then 
		$$\lbal(n,C_{4k+2}) \leq \frac{1}{2}\binom{n}{2} + \theta(kn - k^2).$$
	\end{corollary}
	
	\begin{proof}
		Let $\mathcal{H} = \half(C_{4k+2})$, that is, the family of all union of disjoint paths such that the sum of the lengths of the paths is $2k+1$. Theorem~2 in~\cite{LLP12} states that, for $n$ sufficiently large, the extremal number of $\mathcal{H}$ is the following:
		
		$$ \ex(n,\mathcal{H}) = \left( \sum_{i=1}^\ell \left\lfloor \frac{v_i}{2} \right\rfloor -1 \right) \left( n - \sum_{i=1}^\ell \left\lfloor \frac{v_i}{2} \right\rfloor +1 \right) + \binom{\sum_{i=1}^\ell \left\lfloor \frac{v_i}{2} \right\rfloor -1}{2} + c $$
		
		where $v_i$ (for $i \in \{1,\ldots,\ell\}$) is the order of the $i$th component, and $c=1$ if all of the $v_i$'s are odd and $c=0$ otherwise. Since, in this case, $\sum_{i=1}^\ell v_i = 2k+1+\ell$, we have $\ex(n,\mathcal{H}) = \theta(kn - k^2)$. Hence, \Cref{lem-structure-excluded} yields the statement.
	\end{proof}

    We base the construction of the balanced cycle on the existence of one of two substructures that are called unavoidable patterns and are closely related to the characterization of balanceable graphs; see \cite{CHM19}. In fact, one of these substructures, if it is large enough, naturally contains a balanced $C_{4k+2}$; that is, the color of the bicolored edges may be established before looking for the balanced cycle. However, for the second substructure, the construction of the balanced cycle uses a bicolored edge to leverage the fact that $C_{4k+2}$ is not balanceable; that is, in such case the balanced copy always includes a bicolored edge. 	
	
	
	
	\begin{theorem}
		\label{thm-balancingC4l2}
		Let $k$ be a positive integer. For $n$ sufficiently large, we have $\lbal(n,C_{4k+2})=\frac{1}{2}\binom{n}{2}$.
	\end{theorem}
	
	\begin{proof}
		Since $C_{4k+2}$ is not balanceable, by \Cref{prop-balIsLbal}, we have $\lbal(n,C_{4k+2})\geq \frac{1}{2}\binom{n}{2}$. To prove the equality, we simply have to consider a 2-list edge coloring of $K_n$ inducing color classes $R$ and $B$ where $|R|,|B| \ge \frac{1}{2}\binom{n}{2}+1$ and find a balanced copy of $C_{4k+2}$. Note that there are at least~2 bicolored edges in the 2-list edge coloring. Let $t$ be an integer verifying $t \geq 3k+1$.
		
		
		For the first step, let us ignore the fact that we have bicolored edges: every bicolored edge is set to a fixed color, making sure that both color classes remain balanced and thus contain half ($\pm 1$) the edges of $K_n$. 
		This allows us to apply Theorem 2.1 in~\cite{CHM19}, which ensures that, within $K_n$, there is a copy $H$ of $K_{2t}$ such that there is a partition of its vertex set $V(H) = X \cup Y$ such that $|X| = |Y| = t$, and one of the following hold:
		\begin{itemize}
			\item $E(X) \subseteq R$, and $E(Y) \cup E(X,Y) \subseteq B$ (or vice-versa). We call this a \emph{type-A} copy of $K_{2t}$;
			\item $E(X) \cup E(Y) \subseteq R$, and $E(X,Y) \subseteq B$ (or vice-versa).  We call this a \emph{type-B} copy of $K_{2t}$.
		\end{itemize}
     Let now $e \in R \cap B$ be one of the bicolored edges (the other one will not be needed at all). We prove that whichever type of copy of $K_{2t}$ exists and wherever the bicolored edge $e$ is in $K_n$, we can find a balanced copy of $C_{4k+2}$. There are four cases to consider.
		\ \\
		
		\noindent\textbf{Case 1:} \emph{$H$ is of type-A.} 
		In this case, it is possible to construct the following balanced $(4k+2)$-cycle: follow a red path of length $2k+1$ in $X$, then go to $Y$ through a blue edge, follow a blue path of length $2k-1$ in $Y$, and finally close the cycle by going back to the first vertex that we used in $X$ (using another blue edge). Note that we did not make use of any bicolored edge in this case.
		\ \\
		
		\noindent\textbf{Case 2:} \emph{$H$ is of type-B and $e \in E(H)$.} Then either $e \in E(X)$ (or $e \in E(Y)$, but this is symmetric), or $e \in E(X,Y)$. Let $e = uv$. 
		Both subcases are depicted on \Cref{fig-case2}.
		
		\begin{itemize}
			\item \emph{Subcase 2.1: The bicolored edge $e \in E(X)$.} 
			Let $y \in Y$. We construct a cycle of length $4k+2$ starting with the path $uvy$, following with a red path of length $2k+1$ with all its vertices in $Y$, and then we alternate between $Y$ and $X$ passing through $2k-1$ blue edges and closing the cycle in $u$. This cycle has $2k$ blue edges, $2k+1$ red edges, as well as the bicolored edge $e$. By considering the bicolored edge as being blue, we have a balanced copy of $C_{4k+2}$. This is depicted on \Cref{fig-case21}.
			
			\item \emph{Subcase 2.2: The bicolored edge $e \in E(X,Y)$.} Say $u \in X$ and $v \in Y$. We construct the following cycle: starting from vertex $u$, we go to $v$ through the bicolored edge, then alternate between $Y$ and $X$ following a blue path of length $2k+1$, and complete the cycle with a red path of length $2k$ inside $Y$ that ends in $u$. This cycle has $2k+1$ blue edges, $2k$ red edges, and the bicolored edge $e$. By considering the bicolored edge as being red, we have a balanced copy of $C_{4k+2}$. This is depicted on \Cref{fig-case22}.
		\end{itemize}
		
		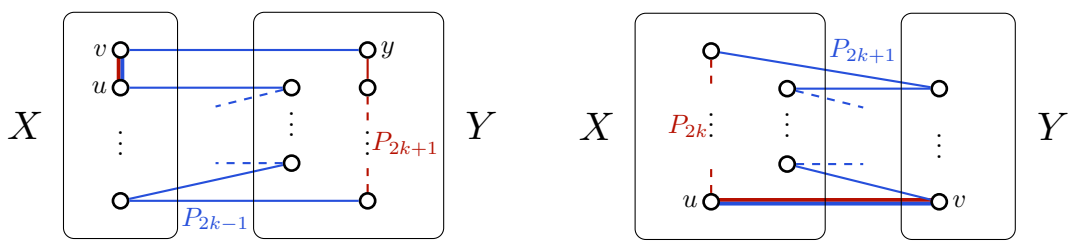
\begin{figure}[h]
			\centering
			\begin{subfigure}{0.45\linewidth}
			\centering
			    \begin{tikzpicture}
            		\draw[rouge,line width=0.5mm] (0.725,2)--(0.725,2.5);
            		\draw[bleu,line width=0.5mm] (0.775,2)--(0.775,2.5);
            		
            		\node[noeud] (u1) at (0.75,0.5) {};
            		\node[noeud] (uk) at (0.75,2) {};
            		\node[noeud] (uk1) at (0.75,2.5) {};
            		
            		\node[noeud] (v2) at (3,1) {};
            		\node[noeud] (vk) at (3,2) {};
            		\node[noeud] (vk1) at (4,2.5) {};
            		\node[noeud] (vk2) at (4,2) {};
            		\node[noeud] (v3k1) at (4,0.5) {};
            		
            		\draw (u1) node[left=0.5mm] {};
            		\draw (uk) node[left=0.5mm] {$u$};
            		\draw (uk1) node[left=0.5mm] {$v$};
            		
            		\draw (v2) node[right=0.5mm] {};
            		\draw (vk) node[right=0.5mm] {};
            		\draw (vk1) node[right=0.5mm] {$y$};
            		\draw (vk2) node[right=0.5mm] {};
            		\draw (v3k1) node[right] {};
            		
            		\draw (0.75,1.375) node {$\vdots$};
            		\draw (3,1.625) node {$\vdots$};
            		\draw (4,1.375) node {$\vdots$};
            		
            		\draw[thick,bleu] (v3k1)to(u1);
            		\draw[thick,bleu] (v2)to(u1);
            		\draw[thick,bleu,dashed] (v2)to(2,1);
            		\draw[thick,bleu,dashed] (vk)to(2,1.75);
            		\draw[thick,bleu] (vk)to(uk);
            		\draw[thick,bleu] (vk1)to(uk1);
            		\draw[thick,rouge] (vk1)to(vk2);
            		\draw[thick,rouge,dashed] (vk2)to(4,1.5);
            		\draw[thick,rouge,dashed] (v3k1)to(4,1);
            		
            		\draw[rounded corners] (0,0) rectangle (1.5,3);
            		\draw[rounded corners] (2.5,0) rectangle (5,3);
            		\draw (-0.5,1.5) node[scale=1.5] {$X$};
            		\draw (5.5,1.5) node[scale=1.5] {$Y$};
            		\draw (4.5,1.25) node[scale=1] {{\color{rouge}$P_{2k+1}$}};
        		    \draw (2,0.25) node[scale=1] {{\color{bleu}$P_{2k-1}$}};
            	\end{tikzpicture}
				\caption{Subcase 2.1: we consider the bicolored edge $x_{k}x_{k+1}$ as having the color $b$.}
				\label{fig-case21}
			\end{subfigure}~~~\begin{subfigure}{0.45\linewidth}
			    \centering
			    \begin{tikzpicture}
            		\draw[rouge,line width=0.5mm] (0.5,0.525)--(3.5,0.525);
            		\draw[bleu,line width=0.5mm] (0.5,0.475)--(3.5,0.475);
            		
            		\node[noeud] (u2) at (1.5,1) {};
            		\node[noeud] (uk2) at (1.5,2) {};
            		\node[noeud] (uk3) at (0.5,2.5) {};
            		\node[noeud] (u3k1) at (0.5,0.5) {};
            		
            		\node[noeud] (v1) at (3.5,0.5) {};
            		\node[noeud] (vk1) at (3.5,2) {};
            		
            		\draw (u2) node[left=0.5mm] {};
            		\draw (uk2) node[above=0.5mm] {};
            		\draw (uk3) node[left=0.5mm] {};
            		\draw (u3k1) node[left=0.5mm] {$u$};
            		
            		\draw (v1) node[right=0.5mm] {$v$};
            		\draw (vk1) node[right=0.5mm] {};
            		
            		\draw (1.5,1.625) node {$\vdots$};
            		\draw (3.5,1.35) node {$\vdots$};
            		\draw (0.5,1.625) node {$\vdots$};
            		
            		\draw[thick,bleu] (v1)to(u2);
            		\draw[thick,bleu] (vk1)to(uk2);
            		\draw[thick,bleu,dashed] (u2)to(2.5,1);
            		\draw[thick,bleu,dashed] (uk2)to(2.5,1.75);
            		\draw[thick,bleu] (vk1)to(uk3);
            		\draw[thick,rouge,dashed] (uk3)to(0.5,2);
            		\draw[thick,rouge,dashed] (u3k1)to(0.5,1);
            	
            		\draw[rounded corners] (-0.5,0) rectangle (2,3);
            		\draw[rounded corners] (3,0) rectangle (4.5,3);
            		\draw (-1,1.5) node[scale=1.5] {$X$};
            		\draw (5,1.5) node[scale=1.5] {$Y$};
            		\draw (2.5,2.5) node[scale=1] {{\color{bleu}$P_{2k+1}$}};
            		\draw (0.2,1.5) node[scale=1] {{\color{rouge}$P_{2k}$}};
            	\end{tikzpicture}
				\caption{Subcase 2.2: we consider the bicolored edge $x_1y_1$ as having the color $r$.}
				\label{fig-case22}
			\end{subfigure}
		\caption{Illustration of \textbf{Case 2} of the proof. The bicolored edge is depicted thick and with both colors.}
		\label{fig-case2}
	\end{figure}
	
	\noindent\textbf{Case 3:} \emph{$H$ is of type-B, and $e = uv$ with $u \in V(H)$ and $v \in V \setminus V(H)$.} Assume, without loss of generality, that $u \in X$. We construct the following path of length $4k$: take a red path of length $2k$ in $X$, and complete it with a blue path of length $2k$ alternating vertices between $X$ and $Y$. Let $w \in X$ be the last vertex of this path. Now, if $vw \in R$ (resp. $vw \in B$), then we consider the bicolored edge as being in $B$ (resp. in $R$). The cycle we constructed has $2k+1$ edges of each color class, and thus it is a balanced copy of $C_{4k+2}$.
	\ \\
	
	\noindent\textbf{Case 4:} \emph{$H$ is of type-B, and $e = uv$ with $u, v \in V \setminus V(H)$.} There are two possible subcases to study. Both subcases are depicted on \Cref{fig-case3}.
	
	\begin{itemize}
		\item \emph{Subcase 4.1: There is a red edge $ux$ (or $vx$) for some $x \in X$.} We construct the following cycle: for any vertex $w \in X \setminus\{x\}$, take the $3$-path $wvux$, then follow with a red path of length $2k-1$ in $X$, and, alternating between $X$ and $Y$, close with a blue path of length $2k$ that ends in $w$.
		
		This cycle has $2k$ red edges and $2k$ blue edges between that are different from $uv$ and $vw$. If $vw$ is red (resp. blue), then we consider $uv$ as being blue (resp. red) and have a balanced copy of $C_{4k+2}$. This is depicted on \Cref{fig-case41}.
		
		\item \emph{Subcase 4.2: All edges $ux$ and $vx$ are blue for every $x \in X$.} For any two vertices $x, x' \in X$, we construct the following cycle: take the $3$-path $x'vux$, continue with a red path of length $2k+1$ in $X$, then alternate vertices between $X$ and $Y$ building a blue path of length $2k-2$ that finishes in $w$ and closes the cycle (if $k = 1$, just take $w$ as the last vertex of the red path).
		
		This cycle has $2k+1$ red edges, $2k$ blue edges and the bicolored edge $e$, that can be considered as being blue. Hence, we have a balanced copy of $C_{4k+2}$. This is depicted on \Cref{fig-case42}.
	\end{itemize}
	
	\begin{figure}[h]
		\centering
		\begin{subfigure}{0.45\linewidth}
		    \centering
		    \begin{tikzpicture}
        		\draw[rouge,line width=0.5mm] (0.5,-1.025)--(1.5,-1.025);
        		\draw[bleu,line width=0.5mm] (0.5,-0.975)--(1.5,-0.975);
        		
        		\node[noeud] (x) at (0.5,-1) {};
        		\node[noeud] (y) at (1.5,-1) {};
        		
        		\node[noeud] (u1) at (0.5,0.5) {};
        		\node[noeud] (u2) at (0.5,1) {};
        		\node[noeud] (u2k) at (0.5,2.5) {};
        		\node[noeud] (u2k1) at (1.5,2) {};
        		\node[noeud] (u3k) at (1.5,0.5) {};
        		
        		\node[noeud] (v1) at (3.5,2.5) {};
        		\node[noeud] (v2) at (3.5,2) {};
        		\node[noeud] (vk) at (3.5,1) {};
        		
        		\draw (x) node[left=0.5mm] {$u$};
        		\draw (y) node[right=0.5mm] {$v$};
        		
        		\draw (u1) node[left=0.5mm] {$x$};
        		\draw (u2) node[left=0.5mm] {};
        		\draw (u2k) node[left=0.5mm] {};
        		\draw (u2k1) node[above=0.5mm] {};
        		\draw (u3k) node[left=0.5mm] {$w$};
        		
        		\draw (v1) node[right=0.5mm] {};
        		\draw (v2) node[right=0.5mm] {};
        		\draw (vk) node[right=0.5mm] {};
        		
        		\draw (0.5,1.875) node {$\vdots$};
        		\draw (1.5,1.5) node {$\vdots$};
        		\draw (3.5,1.625) node {$\vdots$};
        		
        		\draw[thick,rouge] (x)to(u1);
        		\draw[thick,rouge] (u2)to(u1);
        		\draw[thick,rouge,dashed] (u2)to(0.5,1.5);
        		\draw[thick,rouge,dashed] (u2k)to(0.5,2);
        		\draw[thick,bleu] (u2k)to(v1);
        		\draw[thick,bleu] (u2k1)to(v1);
        		\draw[thick,bleu] (u2k1)to(v2);
        		\draw[thick,bleu,dashed] (v2)to(2.5,1.75);
        		\draw[thick,bleu,dashed] (vk)to(2.5,1);
        		\draw[thick,bleu] (u3k)to(vk);
        		\draw (u3k)to(y);
        	
        		\draw[rounded corners] (-0.5,0) rectangle (2,3);
        		\draw[rounded corners] (3,0) rectangle (4.5,3);
        		\draw (-1,1.5) node[scale=1.5] {$X$};
        		\draw (5,1.5) node[scale=1.5] {$Y$};
        		\draw (0,1.5) node[scale=1] {{\color{rouge}$P_{2k-1}$}};
        		\draw (2.5,2.8) node[scale=1] {{\color{bleu}$P_{2k}$}};
        	\end{tikzpicture}
			\caption{Subcase 4.1: we consider the bicolored edge $uv$ as being in a color class different from $x_{3k}u$.}
			\label{fig-case41}
		\end{subfigure}~~~\begin{subfigure}{0.45\linewidth}
		    \centering
		    \begin{tikzpicture}
        		\draw[rouge,line width=0.5mm] (0.5,-1.025)--(1.5,-1.025);
        		\draw[bleu,line width=0.5mm] (0.5,-0.975)--(1.5,-0.975);
        		
        		\node[noeud] (x) at (0.5,-1) {};
        		\node[noeud] (y) at (1.5,-1) {};
        		
        		\node[noeud] (u1) at (0.5,0.5) {};
        		\node[noeud] (u2) at (0.5,1) {};
        		\node[noeud] (u2k) at (0.5,2.5) {};
        		\node[noeud] (u2k1) at (1.5,2) {};
        		\node[noeud] (u3k) at (1.5,0.5) {};
        		
        		\node[noeud] (v1) at (3.5,2.5) {};
        		\node[noeud] (v2) at (3.5,2) {};
        		\node[noeud] (vk) at (3.5,1) {};
        		
        		\draw (x) node[left=0.5mm] {$u$};
        		\draw (y) node[right=0.5mm] {$v$};
        		
        		\draw (u1) node[left=0.5mm] {$x$};
        		\draw (u2) node[left=0.5mm] {};
        		\draw (u2k) node[left] {};
        		\draw (u2k1) node[above=0.5mm] {};
        		\draw (u3k) node[left=0.5mm] {$w$};
        		
        		\draw (v1) node[right=0.5mm] {};
        		\draw (v2) node[right=0.5mm] {};
        		\draw (vk) node[right=0.5mm] {};
        		
        		\draw (0.5,1.875) node {$\vdots$};
        		\draw (1.5,1.5) node {$\vdots$};
        		\draw (3.5,1.625) node {$\vdots$};
        		
        		\draw[thick,bleu] (x)to(u1);
        		\draw[thick,rouge] (u2)to(u1);
        		\draw[thick,rouge,dashed] (u2)to(0.5,1.5);
        		\draw[thick,rouge,dashed] (u2k)to(0.5,2);
        		\draw[thick,bleu] (u2k)to(v1);
        		\draw[thick,bleu] (u2k1)to(v1);
        		\draw[thick,bleu] (u2k1)to(v2);
        		\draw[thick,bleu,dashed] (v2)to(2.5,1.75);
        		\draw[thick,bleu,dashed] (vk)to(2.5,1);
        		\draw[thick,bleu] (u3k)to(vk);
        		\draw[thick,bleu] (u3k)to(y);
        		
        		\draw[rounded corners] (-0.5,0) rectangle (2,3);
        		\draw[rounded corners] (3,0) rectangle (4.5,3);
        		\draw (-1,1.5) node[scale=1.5] {$X$};
        		\draw (5,1.5) node[scale=1.5] {$Y$};
        		\draw (0,1.5) node[scale=1] {{\color{rouge}$P_{2k+1}$}};
        		\draw (2.5,2.8) node[scale=1] {{\color{bleu}$P_{2k-2}$}};
        	\end{tikzpicture}
		    \caption{Subcase 4.2: we consider the bicolored edge $uv$ as having the color $b$.}
		    \label{fig-case42}
		\end{subfigure}
		\caption{Illustration of \textbf{Case 4} of the proof on $C_{14}$. The bicolored edge is depicted thick and with both colors.}
		\label{fig-case3}
	\end{figure}
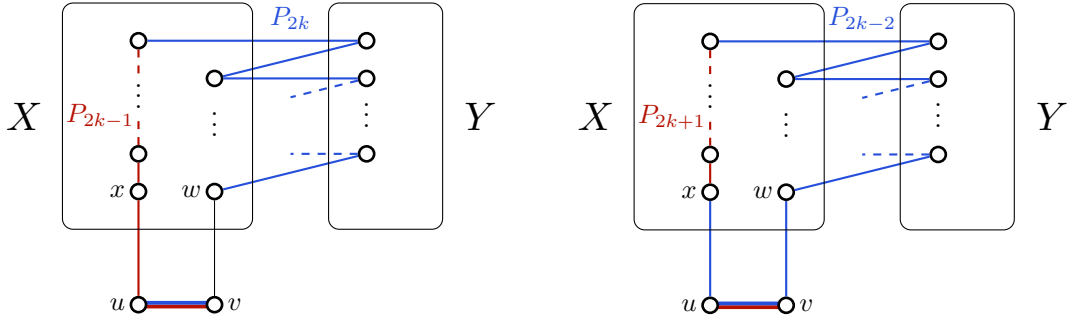
	
	All the cases have been covered: if there is a bicolored edge in $K_n$, then we can find a balanced copy of $C_{4k+2}$, which proves the result.
	\end{proof}
	
	Note that there is an important difference between the upper bound given by \Cref{lem-structure-excluded} and stated in \Cref{cor-balancingC4l2}, and the exact value of the list balancing number as stated in \Cref{thm-balancingC4l2}: the general upper bound gives us a sufficient condition for the existence of $\theta(kn-k^2)$ bicolored edges to guarantee a balanced $C_{4k+2}$, but we actually need only a list-color excess of~1. 

		\section{The list balancing number of $K_5$}
	\label{sec-balancingK5}
	
	Using the characterization of balanceable graphs, it was proved that $K_5$ is not balanceable~\cite{CHM19-2}.
	In this section, we provide lower and upper bounds for the list balancing number of $K_5$; surprisingly, these bounds are matching up to the relevant term. To clarify, trivially, $\lbal(n,K_5)\ge \frac{1}{2}\binom{n}{2}$ and the estimates we obtain in the next theorem have an additional term of order $\ex(n,\{C_3,C_4,C_5\})=\theta(n^{\frac{3}{2}})$. This implies that guaranteeing a balanced copy of $K_5$ requires a remarkably high list-color  excess, implying that we need always a very high amount of bicolored edges. 
	
	\begin{theorem}
		\label{thm-balancingK5}
		Let $c=2\left(\frac{\sqrt{2}-1}{2\sqrt{2}}\right)^{\frac{5}{2}}$. For any $\varepsilon>0$ and $n$ sufficiently large, we have 
		$$\frac{1}{2}\binom{n}{2} +  (1-\varepsilon) c n^{\frac{3}{2}} \le \lbal(n,K_5) \le \frac{1}{2}\binom{n}{2} +  (1+\varepsilon)\frac{1}{4\sqrt{2}}n^{\frac{3}{2}}.$$
	\end{theorem}
	
	Observe that $c\approx 0.016$ while $\frac{1}{4\sqrt{2}}\approx 0.177$. The proof of \Cref{thm-balancingK5} follows directly from \Cref{lem-balancingK5-upperBound} and \Cref{lem-balancingK5-lowerBound} that we state and prove below. For both arguments we focus on the structure of the graph induced by the bicolored edges, where we take into account the girth and the edge number. Recall that, for a graph $G$, the length of a smallest cycle in $G$ is called the girth and is denoted by $g(G)$; if $G$ has no cycles, then its girth is defined to be infinity. Throughout this section we rely on $\ex(n,\{C_3,C_4,C_5\})$, the extremal number for graphs of girth at least 6;
	more precisely, we exploit that $\ex(n,\{C_3,C_4,C_5\})$ is strictly increasing on $n$ and that 
	\begin{equation}\label{eq-ex_small_cycles}
	    \ex(n,\{C_3,C_4,C_5\})= (1+o(1))\frac{1}{2 \sqrt{2}}n^{\frac{3}{2}},
	\end{equation}
	where the asymptotic expression is given in Theorem 4.5 of \cite{FS13}. To verify that 
	$\ex(n,\{C_3,C_4,C_5\})$ is strictly increasing, suppose that $m = \ex(n,\{C_3,C_4,C_5\})$, and take a graph $G$ on $n-1$ vertices and $m-1$ edges with girth at least 6. Then we may construct a graph $G'$ on $n$ vertices and $m$ edges with girth at least 6 by just adding to $G$ a new vertex connected by an edge to any of the vertices in $G$. This proves that $m \le \ex(n,\{C_3,C_4,C_5\})$ and so $\ex(n-1,\{C_3,C_4,C_5\})\le ex(n,\{C_3,C_4,C_5\})- 1$. By iteratively applying this argument, it follows more generally that $\ex(n-k,\{C_3,C_4,C_5\}) \leq \ex(n,\{C_3,C_4,C_5\})-k$.
	
	For the upper bound, we make use of \Cref{lem-structure-excluded}, which boils down to analysing $\ex(n,\half(K_5))$; this is done in the following theorem,
	where we show that $\ex(n,\half(K_5)) = \ex(n,\{C_3,C_4,C_5\})$. 
	
    \begin{theorem}
        \label{thm-ex-half=ex-girth6}
        For $n\ge 5$, we have $\ex(n,\half(K_5))  = \ex(n,\{C_3,C_4,C_5\})$.
	\end{theorem}

	\begin{proof}
		Let $\mathcal{H} = \half(K_5)$, that is, the family of subgraphs of $K_5$ that have 5 edges and no isolates. Observe that $\mathcal{H}$ contains precisely six graphs; namely, the 5-cycle, the 4-pan\footnote{The $n$-pan is an $n$-cycle with a pendant edge attached to a vertex of the cycle.} (also called $P$, or the banner), its complementary $\overline{P}$, the bull, the cricket and the diamond. Those are depicted on \Cref{fig-halfK5}. 
		
		\begin{figure}[!h]
        	\centering
        	\begin{tikzpicture}
        		\node (c5) at (0,0) {
        			\begin{tikzpicture}
        				\node[noeud] (0) at (18:0.75) {};
        				\node[noeud] (1) at (90:0.75) {};
        				\node[noeud] (2) at (162:0.75) {};
        				\node[noeud] (3) at (234:0.75) {};
        				\node[noeud] (4) at (306:0.75) {};
        				\draw (0)to(1)to(2)to(3)to(4)to(0);
        				\draw (0,-1.25) node {$C_5$};
        			\end{tikzpicture}
        		};
        		\node (4pan) at (2,0) {
        			\begin{tikzpicture}
        				\node[noeud] (0) at (0:0.75) {};
        				\node[noeud] (1) at (90:0.75) {};
        				\node[noeud] (2) at (180:0.75) {};
        				\node[noeud] (3) at (270:0.75) {};
        				\node[noeud] (4) at (0,0) {};
        				\draw (0)to(1)to(2)to(3)to(0);
        				\draw (0)to(4);
        				\draw (0,-1.25) node {4-pan = $P$};
        			\end{tikzpicture}
        		};
        		\node (4panB) at (4,0) {
        			\begin{tikzpicture}
        				\node[noeud] (0) at (0,-0.75) {};
        				\node[noeud] (1) at (0.75,-0.75) {};
        				\node[noeud] (2) at (0.375,-0.25) {};
        				\node[noeud] (3) at (0.375,0.25) {};
        				\node[noeud] (4) at (0.375,0.75) {};
        				\draw (0)to(1)to(2)to(0);
        				\draw (2)to(3)to(4);
        				\draw (0.375,-1.25) node {$\overline{P}$};
        			\end{tikzpicture}
        		};
        		\node (bull) at (6,0) {
        			\begin{tikzpicture}
        				\node[noeud] (0) at (0,0) {};
        				\node[noeud] (1) at (0.75,0) {};
        				\node[noeud] (2) at (0.375,-0.5) {};
        				\node[noeud] (3) at (0,0.5) {};
        				\node[noeud] (4) at (0.75,0.5) {};
        				\draw (0)to(1)to(2)to(0);
        				\draw (0)to(3);
        				\draw (1)to(4);
        				\node (inv) at (0,0.75) {};
        				\draw (0.375,-1.25) node {bull};
        			\end{tikzpicture}
        		};
        		\node (cricket) at (8,0) {
        			\begin{tikzpicture}
        				\node[noeud] (0) at (0,0.25) {};
        				\node[noeud] (1) at (0.75,0.25) {};
        				\node[noeud] (2) at (0.375,-0.25) {};
        				\node[noeud] (3) at (-0.25,-0.25) {};
        				\node[noeud] (4) at (1,-0.25) {};
        				\draw (0)to(1)to(2)to(0);
        				\draw (2)to(3);
        				\draw (2)to(4);
        				\node (inv) at (0,0.75) {};
        				\draw (0.375,-1.25) node {cricket};
        			\end{tikzpicture}
        		};
        		\node (diamond) at (10,0) {
        			\begin{tikzpicture}
        				\node[noeud] (0) at (0:0.75) {};
        				\node[noeud] (1) at (90:0.75) {};
        				\node[noeud] (2) at (180:0.75) {};
        				\node[noeud] (3) at (270:0.75) {};
        				\draw (0)to(1)to(2)to(3)to(0);
        				\draw (0)to(2);
        				\draw (0,-1.25) node {diamond};
        			\end{tikzpicture}
        		};
        	\end{tikzpicture}
        	\caption{The family $\half(K_5)$.}
        	\label{fig-halfK5}
        \end{figure}
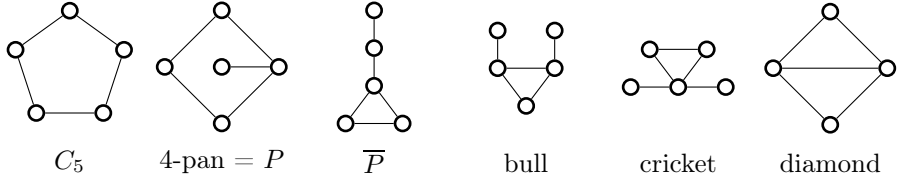
		Observe that every graph from $\mathcal{H}$ has either a $C_3$, a $C_4$, or a $C_5$. Hence, the class of graphs of order $n$ having girth at least 6 is contained in the class of the $\mathcal{H}$-free graphs of order $n$. This implies directly that $\ex(n,\half(K_5)) \ge \ex(n,\{C_3,C_4,C_5\})$. 
		
	    We will prove now the other inequality, that is, that every graph on $n$ vertices and more than $\ex(n,\{C_3,C_4,C_5\})$ edges contains a subgraph from $\mathcal{H}$. 
		
		We use and induction argument. 
		Let $b=\ex(n,\{C_3,C_4,C_5\})+1$. Let $F$ be a graph on $n$ vertices and with at least $b$ edges. We will prove that $F$ contains a subgraph in $\mathcal{H}$. We start with the base cases $n \in \{5,6,7,8\}$:
		\begin{enumerate}
			\item If $n=5$, then any spanning tree is a maximal graph of girth at least~6, and thus $b=4+1=5$. This implies that $F$ contains at least~5 edges, and since $n=5$, $F$ must have a subgraph in $\mathcal{H}$.
			\item If $n=6$, then the maximal graph of girth at least~6 is $C_6$, and thus $b=6+1=7$. Let $F'$ be a subgraph of $F$ on exactly $7$ edges. Suppose every set of $5$ vertices in $F'$ induces a graph of at most $4$ edges. Since $e(F') =7$, the vertex not contained in certain $5$-set has to have degree at least $3$. But this happens to every set of $5$ vertices. Hence, $2e(F') \ge 6 \cdot 3 = 18$, implying that $e(F') \ge 9$, a contradiction. Hence, there is a $5$-set inducing a graph on at least $5$ edges in $F'$ and thus in $F$, and so $F$ contains a subgraph from $\mathcal{H}$.
			\item If $n=7$, then the maximal graphs of girth at least~6 are $C_7$ and the $6$-pan, and thus $b=7+1=8$.  Let $F'$ be a subgraph of $F$ on exactly~8 edges. Observe that $F'$ contains at least an induced cycle of length at most~5. If $F'$ contains an induced $C_5$, then it trivially contains a subgraph from $\mathcal{H}$. If $F'$ contains an induced $C_4$, then since there are at least four remaining edges and only three remaining vertices, this implies that at least one vertex from the $4$-cycle has a neighbour among the other three vertices, which in turn implies that $F'$ contains a 4-pan, which is in $\mathcal{H}$. If $F'$ contains a triangle, then there are three cases: first, there are at least two edges between the triangle and the remaining vertices, and $F'$ contains a bull, a cricket, or a diamond, which are in $\mathcal{H}$; second, there is no edge between the triangle and the 4 remaining vertices, which implies that they must induce a diamond, which is in $\mathcal{H}$; finally, if there is exactly one edge between the triangle and one of the remaining vertices, say $u$, then $u$ has to have a neighbour in the other remaining vertices (since otherwise there would be four edges among three vertices, which is impossible), and $F'$ contains the complement of a 4-pan, which is in $\mathcal{H}$.
			
			Hence, in every case, $F'$, and thus $F$, contains a graph of $\mathcal{H}$ as a subgraph. 
			\item If $n=8$, then the maximal graph of girth at least~6 contains 9 edges (it consists in vertices $a,b,c,d,e,f,e',f'$ that are arranged in two cycles $abcdefa$ and $abcde'f'a$), and thus $b=9+1=10$. Let $F'$ be a subgraph of $F$ on exactly~10 edges. Then $F'$ contains an induced cycle of length at most 5. If $F$ contains an induced $C_5$, then it trivially contains a subgraph from $\mathcal{H}$. If $F'$ contains an induced $C_4$, then there are two cases: either there is at least one edge between the 4-cycle and the remaining vertices, and thus $F'$ contains a 4-pan, which is in $\mathcal{H}$; or the four remaining vertices have to induce a diamond, which is in $\mathcal{H}$. If $F'$ contains a triangle, then there are two cases: either there are at least~2 edges between the triangle and the remaining vertices, and thus $F'$ contains either a bull or a cricket, which are in $\mathcal{H}$; or the five remaining vertices have at least~6 edges, and by the argument in the case $n=5$ implies that $F'$ contains a subgraph in $\mathcal{H}$.
			
			Hence, in every case, $F'$, and thus $F$, contains a graph of $\mathcal{H}$ as a subgraph. 
		\end{enumerate}
		
		For the induction step, we will use the following general argument. Suppose that $F$ is a graph on $n$ vertices and at least $b=\ex(n,\{C_3,C_4,C_5\})+1$ edges; if $F'$ may be constructed from $F$ by removing $k$ vertices and $k$ edges, then $F'$ has (also) girth at most 5. To see this, recall that $\ex(n,\{C_3,C_4,C_5\})$ is strictly increasing, so that since $F'$ has at least $b'=b-k$ edges satisfying $b'>\ex(n,\{C_3,C_4,C_5\})-k \ge \ex(n-k,\{C_3,C_4,C_5\})$. If $n\ge 9$ and $k\le 4$ we may apply the induction hypothesis and infer that $F'$ contains a subgraph in $\mathcal{H}$. In what follows we refer to this argument as the \emph{removal induction hypothesis}.
		
		Now, assume that $n \geq 9$. First, if $F$ contains a vertex of degree $1$, then we can remove it and apply the removal induction hypothesis. Thus, we may assume that $F$ has minimum degree at least 2.
		
		Since $b > \ex(n,\{C_3,C_4,C_5\})$, we know that $F$ has girth at most~5. There are three cases to consider:
		
	\noindent\textbf{Case 1:} \emph{$F$ has girth~5.} Naturally, $F$ contains a subgraph in $\mathcal{H}$; namely, $C_5$.
	
	\noindent\textbf{Case 2:} \emph{$F$ has girth~4.} We may consider a $C_4$ in $F$. If all four vertices have degree~2, then we can remove them from $F$ and apply the removal induction hypothesis. Otherwise, at least one of them has a third neighbour; the cycle together with such neighbor forms a 4-pan, that is $F$ contains a subgraph of $\mathcal{H}$.
	
	\noindent\textbf{Case 3:} \emph{$F$ has girth~3.} We may consider a $C_3$ in $F$. If all three vertices have degree~2, then likewise we can apply the removal induction hypothesis. Otherwise, at least one of them has a third neighbour $u$. However, $u$ has degree at least~2, so it itself has another neighbour; we then obtain either the diamond or the complement of the 4-pan as a subgraph, both of which are in $\mathcal{H}$.

	This proves that $F$ contains a subgraph in $\mathcal{H}$ whenever it has more than $\ex(n,\{C_3,C_4,C_5\})$ edges. Hence $\ex(n,\mathcal{H})\le \ex(n,\{C_3,C_4,C_5\})$ and, thus, we can apply \Cref{lem-structure-excluded}, which proves the statement.
	\end{proof}

    By combining \Cref{lem-structure-excluded} and \Cref{thm-ex-half=ex-girth6}, we obtain the desired upper bound on $\lbal(n,K_5)$.
    
	\begin{corollary}
	\label{lem-balancingK5-upperBound}
	For any $\varepsilon > 0$ and $n$ sufficiently large, 
	\[\lbal(n,K_5) \leq \frac{1}{2}\binom{n}{2}+ (1+\varepsilon)\frac{1}{4 \sqrt{2}}n^{\frac{3}{2}}.\]
	\end{corollary}
    
    \begin{proof}
    Let $\varepsilon > 0$. For $n$ sufficiently large, we have with \Cref{lem-structure-excluded}, \Cref{thm-ex-half=ex-girth6} and (\ref{eq-ex_small_cycles}) that 
    \[\lbal(n,K_5) \leq \frac{1}{2}\binom{n}{2}+\left\lceil \frac{1}{2} \ex(n,\half(K_5))\right\rceil = \frac{1}{2}\binom{n}{2}+\left\lceil \frac{1}{2} \ex(n,\{C_3,C_4,C_5\}) \right\rceil \le \frac{1}{2}\binom{n}{2}+ (1+\varepsilon)\frac{1}{4 \sqrt{2}}n^{\frac{3}{2}}.\]
    \end{proof}
	
	We will now obtain a lower bound for the list balancing number of $K_5$. 	In \Cref{lem-balancingK5-lowerBound}, we provide a 2-list edge coloring of $K_n$ where the subgraph induced by the bicolored edges is of girth at least 6. By analyzing all possible overlaps of a copy of $K_5$ and the bicolored edges, we prove that this 2-list edge coloring does not contain a balanced copy of $K_5$.
		
	\begin{lemma}
		\label{lem-balancingK5-lowerBound}
		Let $c=2\left(\frac{\sqrt{2}-1}{2\sqrt{2}}\right)^{\frac{5}{2}}$. For any $\varepsilon>0$ and $n$ sufficiently large, $$\lbal(n,K_5)\geq  \frac{1}{2}\binom{n}{2} +  (1-\varepsilon)c n^{\frac{3}{2}}.$$
		
	\end{lemma}

	\begin{proof}
	    Suppose that there are integers $k,k'$ and $m$ such that $k\le k'\le n$ and that there exist a graph $H$ on $k$ vertices, $m$ edges and girth at least 6. The precise values for these integers will be specified, in terms of $n$ and $\varepsilon>0$ further on. First, using the assumptions above, we construct a 2-list edge coloring on $K_n$ and prove that it does not contain a balanced copy of $K_5$.
	    
	    Let us partition the vertices of $K_n$ in two parts $X$ and $Y$ such that $|Y|=k'$ (and thus $|X|=n-k'$); assign the list $\{r\}$ to every edge within $X$; assign the list $\{r,b\}$ to $m$ edges in $Y$ inducing a copy of $H$; and finally assign the list $\{b\}$ to every other edge within $Y$ and to every edge between $X$ and $Y$. 
	    
	    We claim that no copy of $K_5$ can be balanced in this coloring. First, any copy of $K_5$ with all its vertices in $X$ has no blue edges and, thus, it cannot be balanced. Now, let $G$ be a copy of $K_5$ with at least one vertex in $Y$ and let $x$ and $y$ be the number of vertices of $G$ in $X$ and $Y$, respectively; note that $y\ge 1$. Recall that bicolored edges form a graph of girth at least 6 and so $G$ has at most $y-1$ bicolored edges in $G$ and precisely $\binom{x}{2}$ red (non-bicolored) edges. To conclude the proof that there is no balanced copy of $K_5$ observe that if $x\ge 4$ then $G$ has at most 4 blue edges, including bicolored ones. Whereas if $x\le 3$, then $G$ has at most $\binom{x}{2}+y-1\le x+y-1=4$ red edges, including bicolored ones; thus, $G$ may not be balanced.

        It remains to prove that, given $\varepsilon>0$, we may choose $k,k'$ and $m$ so that the color classes of the coloring above have size at least $\frac{1}{2}\binom{n}{2}+(1-\varepsilon)c n^{\frac{3}{2}}$; as this would establish the lemma.
        
        Fix $\varepsilon>0$ and let $\alpha=1-\frac{1}{\sqrt{2}}$, $\beta=\left(1-\frac{\varepsilon}{2}\right)\left(\frac{\alpha}{2}\right)^{\frac{3}{2}}$; then let $k=\lceil \alpha n\rceil$, $k'=\left\lceil \alpha n + \beta n^{\frac{1}{2}}\right\rceil$ and  $m=\left\lfloor \beta n^{\frac{3}{2}}\right\rfloor$. 
        Observe that 
        $$m=\left\lfloor \beta n^{\frac{3}{2}} \right\rfloor \le \left(1-\frac{\varepsilon}{2}\right)\left(\frac{\alpha}{2}\right)^{\frac{3}{2}}n^{\frac{3}{2}} \le \left(1-\frac{\varepsilon}{2}\right)\left(\frac{k}{2}\right)^{\frac{3}{2}}\le  \ex(k,\{C_3,C_4,C_5\}); $$
        where the last inequality holds for $n$ large enough since $\ex(k,\{C_3,C_4,C_5\})=(1+o(1))(\frac{k}{2})^{\frac{3}{2}}$ by Theorem 4.5 of \cite{FS13}. This establishes the existence of a graph $H$ with girth at least 6, as desired. Moreover, we have clearly $k\le k'\le n$.
        Next, we will show that $|R|,|B|>\frac{n^2}{4}+\left(1-\frac{\varepsilon}{2}\right)\alpha\beta  n^{\frac{3}{2}}$. 
		
		In the following expressions we assume that $n$ is large enough that we may omit rounding to integers to avoid cumbersome notation; in particular we will simply write $n-k'=(1-\alpha)n-\beta n^{\frac{1}{2}}$ (To clarify, considering the precise expression of $n-k'$ would only add, to $|R|$ and $|B|$, terms of order $O(n)$ which may be neglected).
		
		We clearly have $|R|=\binom{n-k'}{2}+m$ and $|B|=\binom{k'}{2}+k(n-k')$.
 		First, we consider the size of $R$; using that $m=\beta n^{\frac{2}{3}}$, we obtain
		\begin{align*}
		    \binom{n-k'}{2} +m 
		    &= 
		    \frac{1}{2}\left((1-\alpha)n-\beta n^{\frac{1}{2}})^2-(1-\alpha)n+\beta n^{\frac{1}{2}}\right)+\beta n^{\frac{3}{2}}\\
		    &=
		    \frac{(1-\alpha)^2n^2}{2} +\alpha\beta n^{\frac{3}{2}}+\frac{(\beta^2+\alpha-1)n}{2} +\frac{\beta n^{\frac{1}{2}}}{2}\\
		    &>  \frac{n^2}{4} +\alpha\beta n^{\frac{3}{2}}-\frac{n}{2};
       \end{align*}
       where in the last inequality we used that $1-\alpha=\frac{1}{\sqrt{2}}$ and removed lower order positive terms. In addition, we have that $\frac{\varepsilon\alpha\beta}{2} \ge n^{-\frac{1}{2}}$ for $n$ large enough, and so
       $$\alpha\beta n^{\frac{3}{2}}-\frac{n}{2}=\left(1-\frac{\varepsilon}{2}\right)\alpha\beta  n^{\frac{3}{2}}+ n^{\frac{3}{2}}\left(\frac{\varepsilon\alpha\beta}{2}-n^{-\frac{1}{2}}\right)> \left(1-\frac{\varepsilon}{2}\right)\alpha\beta n^{\frac{3}{2}};$$ which in turn implies that $|R|>\frac{n^2}{4}+(1-\frac{\varepsilon}{2})\alpha \beta n^{\frac{3}{2}}$ for $n$ sufficiently large. 
       Similar computations for the size of $B$ yield
	\begin{align*}
		    \binom{k'}{2}+k'(n-k')&=\frac{1}{2} \left((\alpha n+\beta n^{\frac{1}{2}})^2-(\alpha n+\beta n^{\frac{1}{2}})\right) +\left(\alpha n+\beta n^{\frac{1}{2}}\right)\left((1-\alpha)n-\beta n^{\frac{1}{2}}\right) \\
		    &=\left(\frac{\alpha^2}{2}+\alpha(1-\alpha)\right)n^2+\beta(1-\alpha) n^{\frac{3}{2}}-\frac{(\beta^2-2\alpha)n}{2}-\frac{\beta}{2} n^{\frac{1}{2}}\\
		    &> \frac{n^2}{4}+\beta(1-\alpha) n^{\frac{3}{2}}-\frac{\beta^2 n}{2}-\frac{\beta}{2} n^{\frac{1}{2}}.
		\end{align*}
		In this case we use that $2-4\alpha=\frac{4}{\sqrt{2}}-2>0$, and so for $n$ large enough we have
		\begin{align*}
		    \beta(1-\alpha) n^{\frac{3}{2}}-\frac{\beta^2 n}{2}-\beta n^{\frac{1}{2}}
		    =\alpha\beta n^{\frac{3}{2}} + \frac{\beta n^{\frac{3}{2}}}{2}(2-4\alpha-\beta n^{-\frac{1}{2}}- n^{-1})
		    > \alpha \beta n^{\frac{3}{2}}.
		\end{align*}
		Finally, observe that $\frac{n^2}{4}\ge \frac{1}{2}\binom{n}{2}$; on the other side, $(1-\frac{\varepsilon}{2})\alpha\beta\ge (1-\varepsilon)2(\frac{\alpha}{2})^{\frac{5}{2}}=(1-\varepsilon)c$ and so we conclude that, for $n$ large enough, $$|R|,|B|> \frac{n^2}{4}+\left(1-\frac{\varepsilon}{2}\right)\alpha \beta n^{\frac{3}{2}}\ge \frac{1}{2}\binom{n}{2}+ (1-\varepsilon)c n^{\frac{3}{2}};$$ as desired.
	\end{proof}

	\section{Conclusion}
	
	In this paper, we studied the balancing number and the list balancing number of several graph classes. First, we found the exact value for the balancing number of odd cycles and gave upper and lower bounds for the balancing number of $C_{4k}$ which are tight up to first order terms. The proofs are based on the following idea: from a balanced path, we can construct a balanced cycle. We believe that the lower bound obtained by the construction given in the proof of \Cref{lem-balC4k-lowerBound} is tight and that the upper bound obtained in \Cref{lem-balC4k-upperBound} could be improved by means of a carefully analysis of the color patterns inside the set $W$. 
	
	
	We also introduced the list balancing number, an extension of the balancing number. We did this by allowing edges to belong to both color classes, by way of replacing the 2-edge coloring of $K_n$ by a 2-list edge coloring. The goal is to understand exactly if non-balanceable graphs are, in a way, close or far from being balanceable. For example, we only need a list-color excess of one to guarantee a balanced copy of $C_{4k+2}$, while way more ($\theta(\ex(n,\{C_3,C_4,C_5\}))$, which is in $O(n^{\frac{3}{2}})$) are required to guarantee a balanced $K_5$. Furthermore, while we determined a general upper bound for the list balancing number of a graph $G$, based on the extremal number of subgraphs of $G$ containing at least half the edges of $G$, this bound can be arbitrarily bad (as is the case for $C_{4k+2}$). Hence, this extension opens many interesting questions: for which graph classes is the general upper bound good? For those in which it is bad, what is the exact value of the list balancing number? Which graphs are close to being balanceable, like $C_{4k+2}$?
	

\section*{Acknowledgements}
We would like to thank BIRS-CMO for hosting the workshop Zero-Sum Ramsey Theory: Graphs, Sequences and More (19w5132), where many fruitful discussions arose that contributed to a better understanding of these topics.\\

\end{document}